\documentclass[11pt]{amsart}
\usepackage{amsfonts}
\usepackage{amsfonts,latexsym,rawfonts,amsmath,amssymb,amsthm}
\usepackage[plainpages=false]{hyperref}

\usepackage{graphicx}
\makeatletter
\renewcommand\@biblabel[1]{}
\makeatother

\numberwithin{equation}{section}
\newcommand{\beq}{\begin{equation}}
\newcommand{\eeq}{\end{equation}}
\newcommand{\beqs}{\begin{eqnarray*}}
\newcommand{\eeqs}{\end{eqnarray*}}
\newcommand{\beqn}{\begin{eqnarray}}
\newcommand{\eeqn}{\end{eqnarray}}
\newcommand{\beqa}{\begin{array}}
\newcommand{\eeqa}{\end{array}}

\def\lra{\longrightarrow}

\def\bc{\begin{center}}
\def\ec{\end{center}}

\def\begeq{\begin{equation}}
\def\endeq{\end{equation}}
\def\and{\quad{\rm and}\quad}

\let\lra=\longrightarrow

\def\mapright\#1{\,\smash{\mathop{\lra}\limits^{\#1}}\,}

\newtheorem{prop}{Proposition}[section]
\newtheorem{theo}[prop]{Theorem}
\newtheorem{lem}[prop]{Lemma}
\newtheorem{claim}[prop]{Claim}
\newtheorem{cor}[prop]{Corollary}
\newtheorem{rem}[prop]{Remark}

\newtheorem{defi}[prop]{Definition}

\title  {Structure of spaces with  Bakry-\'{E}mery  Ricci curvature  bounded below}

\author   {Feng Wang }
\author { Xiaohua $\text{Zhu}^*$}

\thanks {* Partially supported by the NSFC Grants 10990013 and 11271022}
 \subjclass[2000]{Primary: 53C25; Secondary:  53C55,
 58J05}
\keywords { Bakry-\'{E}mery  Ricci curvature,  Gromov-Hausdorff topology,
tangent cone structure}

\address{ Feng Wang\\School of Mathematical Sciences, Peking University,
Beijing, 100871, China}

\address{ Xiaohua Zhu\\School of Mathematical Sciences and BICMR, Peking University,
Beijing, 100871, China\\
 xhzhu@math.pku.edu.cn}

\begin{document}
\bibliographystyle{plain}

\begin{abstract} In this paper,  we explore  the   limit    structure of  a sequence  of Riemannian manifolds with   Bakry-\'{E}mery Ricci curvature bounded  below  in   the Gromov-Hausdorff topology.    By  extending  the techniques  established  by Cheeger-Cloding for Riemannian manifolds with   Ricci curvature bounded  below, we  prove that  each  tangent  space  at a point of  the limit space is a metric cone.    We also analyze  the singular  structure of
the    limit space analogous  to  a  work  of Cheeger-Colding-Tian.  Our results will be  applied to study the limit space 
of  a sequence of  K\"ahler  metrics arising from solutions of  certain complex Monge-Amp\`ere equations  for the existence of K\"ahler-Ricci solitons
  on a Fano manifold  via  the continuity method.

\end{abstract}

\maketitle

\tableofcontents

\setcounter{section}{-1}

\section{Introduction}

In a series of papers [CC1], [CC2], [CC3],  Cheeger-Colding  study  the   limit space of  a sequence  of Riemannian manifolds with   Ricci curvature bounded below  in  the Gromov-Hausdorff topology.  As one of fundamental results, they prove the existence of  metric cone structure for each  tangent cone on  the limit space  [CC2].  Namely,

\begin{theo}([CC2])\label{thm-cc1}
Let $(M_i,g_i;p_i)$ be a sequence of $n$-dimentional  Riemannian manifolds which  satisfy
$${\rm Ric}_{M_i}(g_i)\geq{- (n-1)\Lambda^2g_i}~{\rm and}~{\rm vol}_{g_i}( B_{p_i}(1))\geq v>0.$$
  Then  $(M_i,g_i;p_i)$ converge to  a metric space   $(Y;p_\infty)$  in the pointed  Gromov-Hausdorff topology.  Moreover, for any $y\in Y$,  each tangent cone
 $T_yY$ is a metric cone over another metric space whose diameter is less than $\pi$.
 \end{theo}

Based on the above theorem, Cheeger-Colding   introduce  a notion of $\mathcal S_k$-typed ($k\le n-1$)  singularities of  the
 limit space $Y$ as follows.

\begin{defi}\label{singular-type}  Let $(Y;p_{\infty})$ be the  limit of $(M_i,g_i;p_i)$ as in Theorem \ref{thm-cc1}.  We call $y\in (Y;p_\infty)$  a  $\mathcal S_k$-typed singular point if  there exists  a tangent cone at $y$ which can be  split   out an euclidean space   $\mathbb R^k$ isometrically  with dimension  at most  $k$.
\end{defi}

Applying  Theorem \ref{thm-cc1} to appropriate tangent cone spaces,   Cheeger-Colding show that
  the  dimension of  set $\mathcal S_k$   is less than $k$ [CC2].   In  [CCT],  Cheeger-Colding-Tian  do  a significant  work  to determine which kind of singularities  can be excluded  in the limit space $Y$  under certain curvature condition for the sequence of $(M_i,g_i)$.  They  prove

\begin{theo}([CCT])\label{thm-cct}
Let $(M_i, g_i;  p_i)$ be a sequence of $n$-dimensional manifolds  and $(Y,p_\infty)$    its  limit as in Theorem  \ref{thm-cc1}.
Suppose that the integrals  of sectional  curvature
$$\frac{1}{{\rm vol}_{g_i}(B_{p_i}(1))}\int_{B_{p_i}(1)}|{\rm Rm}|^p d{\rm v}$$
are   uniformly bounded.  Then  for any $\epsilon>0$, the following is true:
$${\rm dim} (B_{p_\infty}(1)\setminus R_\epsilon)\leq n-4,~\text{ if} ~p=2$$
 and
 $$\mathcal H^{n-2p}(B_{p_{\infty}}(1)\setminus R_\epsilon)<\infty,  ~\text{if}~ 1\leq p<2.$$
   Here  $R_\epsilon$  consists  of  points  $y$ in $Y$  which satisfy
$${\rm dist}_{GH}(B_y(1), B_0(1))\leq\epsilon$$
 for   the unit ball   $ B_0(1) $ in  $\mathbb R^n$ and a unit distance  ball  $B_y(1)$  in some tangent cone $T_yY$.

\end{theo}

  The purpose of the present paper is to  extend   the above Cheeger-Colding theorem and
Cheeger-Colding-Tian theorem   in  the  Bakry-\'Emery geometry.  More precisely,  
we  analyze the  structure of Gromov-Hausdorff   limit  of a sequence of $n$-dimensional  Riemannian manifolds   in  class 
  $\mathcal{M}(A,\Lambda, v)$  which defined by
  \begin{align}
  \mathcal{M} (A,\Lambda,v)=&\{(M,g;p)|~M \text{ is an  $n$-dimensional}\notag\\
  &\text{complete Riemannian manifold which satisfy} \notag\\
   &{\rm Ric}_M(g)+\text{hess}(f)\geq-(n-1)\Lambda^2 g,\notag\\
   &{\rm vol}_g(B_{p}(1))\geq v>0,~
   \text{and}~  |\nabla f|_g\le A\}.\notag
   \end{align}
  Here $f$ is a smooth function on  $M$ and  $\text{hess}(f)$ denotes  Hessian tensor of $f$ with respect to $g$.   ${\rm Ric}_M(g)+\text{hess}(f)$ is called Bakry-\'Emery Ricci curvature associated to $f$ [BE]. For simplicity, we denote it by ${\rm Ric}_{M,g}^f$ or just  ${\rm Ric}_{g}^f$.  Clearly,   $\mathcal{M}(A,\Lambda, v)$ consists of compact Ricci solitons [Ha], [TZh].  We show that
both  Theorem \ref{thm-cc1} and Theorem \ref{thm-cct}  still hold  for a sequence in  $\mathcal{M}(A,\Lambda, v)$   (cf. Section 4, 5).

As in [CC1],   we shall establish  various  integral  comparison results   for  the gradient  and Hessian
estimates  between   appropriate   functions  and coordinate
functions  or  distance functions  on a Riemannian manifold with   Bakry-\'Emery Ricci curvature bounded below.  We will use $f$-harmonic functions   to construct  those appropriate   functions instead of harmonic functions  (cf. Section 2).  Another technique is to generalize  the  segment inequality lemmas   in [CC1]  to our case of  weighted volume form  (cf. Lemma \ref{equ-seg}, Lemma \ref{equ-rad},  Lemma \ref{approxi-1}) so that the  triangle lemmas in [Ch2] are still true on a Riemannian manifold with  almost flat  Bakry-\'Emery Ricci curvature  (cf. Lemma \ref{cheeger-lemma},   Lemma \ref{cheeger-lemma-2}). These triangle lemmas are crucial in  proofs of  the splitting theorem and  the metric cone theorem  (cf. Theorem \ref{splitting-theorem}, Theorem \ref{existence-metric-cone}).   We shall point out that  various  versions of such kind  triangle lemmas  were used   by Colding,  Cheeger-Colding  in earlier  papers  to  study   the rigidity  of  of Riemannian metrics  [Co1],  [Co2],  [CC1].

Another motivation of this paper is to  study the limit  space of   a sequence of  K\"ahler metrics $g_{t_i}$  $(t_i<1)$ arising from  solutions of certain complex Monge-Amp\`ere equations  for the existence 
 of K\"ahler-Ricci  soliton  via  the continuity method [TZ1], [TZ2].  We  show that such metrics are naturally  belonged to   $\mathcal{M}(A,v,\Lambda)$.  As a consequence,   for any sequence  $\{g_{t_i}\}$  there exists a subsequence which converges  to a metric space with  complex codimention of singularities at least one  in the  Gromov-Hausdorff topology  (cf.   Theorem \ref{thm-kahler-1},  Section 6).  Furthermore,  in case of  $t_i\to 1$, 
the  complex codimention of singularities of  limit space is  at least two (cf.   Theorem \ref{thm-kahler-2}). The later  is  corresponding to a sequence of  called weak almost  K\"ahler-Ricc solitons, which is a generalization of 
sequence of  weak almost  K\"ahler-Einstein metrics introduced by    Tian-Wang  in  a recent paper  [TW]  (cf. Definifion \ref{almost-kr-soliton}).  In fact, for such a  kind of K\"ahler metrics sequence, 
 we prove the following result:

\begin{theo}\label{thm-kahler-3}
Let $(M_i, g_{i})$ be a  sequence of
  weak  almost K\"ahler-Ricci solitons.   Suppose that there exists a point $p_i$ at each $M_i$ such that
  \begin{align}{\rm vol}_{M_i}(B_{p_i}(1))\ge  v>0.\end{align}
   Then there exists a subsequence of $(M_i,g_i;p_i)$  which converge to  a  limit metric space $Y$
    in  the pointed Gromov-Hausdorff topology. Moreover $\mathcal{S}(Y)=\mathcal{S}_{2n-4}$.  In particular, the complex codimension of singularities of
    $Y$ is at least 2.
\end{theo}

As a corollary of Theorem \ref{thm-kahler-3},  we show that there exists  a sequence of  weak almost  K\"ahler-Ricc solitons  on $M$ which  converges to a metric space  $(M_\infty, g_\infty)$  with complex codimension of the singular set of $(M_\infty, g_\infty)$  at least two  in  the Gromov-Hausdorff topology  if the modified Mabuchi $K$-energy defined in [TZ1] is bounded from below.  In a sequel of papers [WZ]  and [JWZ],  we will further confirm that   the regular  part of $(M_\infty, g_\infty)$ is in fact a K\"ahler-Ricc soliton while    $(M_\infty, g_\infty)$ admits a $Q$-Fano algebraic structure.  

The organization of paper is as follows.  In Section 1,   we  first recall  a $f$-Lapalace comparison result   of  Wei -Wylie   for distance functions (cf.  Lemma \ref{lapalace-esti-r}).   Then  as applications of  Lemma  \ref{lapalace-esti-r} we  construct a cut-off function with  bounded gradient and $f$-Lapalace (cf. Lemma \ref{cut-off}).
 In Section 2, we  give   various  integral estimates for gradient and  Hessian  of $f$-harmonic functions.
In Section 3  and  Section 4,  we will prove the  splitting theorem (cf. Theorm \ref{splitting-theorem}) and the  metric cone theorem (cf. Theorem \ref{existence-metric-cone}),  respectively.   In  Section 5, we give a generalization  of     Cheeger-Colding-Tian's Theorem  \ref{thm-cct}  in  the  setting of  Bakry-\'{E}mery geometry (cf. Theorem \ref{dimension-n-4}).  In Section 6,  we prove Theorem \ref {thm-kahler-1} and Theorem \ref {thm-kahler-3}.  Section 7 is an appendix  where  we explain how to use the technique of  conformal transformation   in [TZh] to give another proof of Theorem \ref {thm-kahler-1} and Theorem \ref {thm-kahler-2}. Section 8 is another appendix where  the relation (\ref{J-invariant}) in Section 6 is proved.

\noindent {\bf Acknowledgements.}   The authors would like to thank
professor G.  Tian for many valuable discussions on this work.   They  are also appreciated to    professor T. Colding  for his interest to the paper, particularly,   for valuable  comments  on   Lemma \ref{cheeger-lemma} and Lemma \ref{cheeger-lemma-2}.

\vskip3mm

\section{Distance function comparison and  other comparison lemmas}

 The notion of Bakry-\'Emery Ricci curvature  ${\rm Ric}_{M,g}^f$ associated to a smooth function $f$
  on a Riemannian manifold $(M,g)$ was first appeared in [BE].  Related to the conformal geometry,  one can introduce a weighted volume form and a  $f$-Lapalace operator associated to $f$ on $(M,g)$ as follows,
\begin{align}
    d\text{v}^f=e^{-f}d\text{v}~\text{and}~
    \Delta^f=\Delta - \langle\nabla f,\nabla\rangle.\notag
    \end{align}
 Then   $\Delta^f$ is   a   self-adjoint elliptic operator
 under the following weighted inner product,
 $$(u,v)=\int_M uv d\text{v}^f,~~\forall ~u,v\in L^2(M).$$
 That is
$$\int_M \Delta^f u  v d\text{v}^f=\int_M \langle\nabla u,\nabla v\rangle d\text{v}^f=\int_M \Delta^f v u d\text{v}^f.$$
The divergence theorem with respect to  $\Delta^f$ is
$$\int_\Omega \Delta^f u d\text{v}^f=\int_{\partial\Omega} \langle\nabla u,n\rangle e^{-f}d\sigma,$$
where  $\Omega$ is a domain in   $M$ with piece-wise smooth boundary,  $n$ denotes the outer unit normal vector field on  $\partial\Omega$ and $d\sigma$ is an induced area form of $g$
on  $\partial\Omega$.

Let $r=r(x)=\text{dist}(p,x)$ be a distance function on $(M,g)$.   In [WW], Wei-Wylie compute
the  $f$-Laplacian  for $r$ and
   got  the following comparison result under the Bakry-Emery Ricci curvature condition.

\begin{lem}([WW])\label{lapalace-esti-r}
Let $(M,g)$ be  an $n$-dimensional  complete  Riemannian manifold which satisfies
\begin{align}\label{ricci-condition-1} {\rm Ric}_{g}^f\geq-(n-1)\Lambda^2g.\end{align}
Then
\begin{align}\label{lapalace-r-1}
\Delta^fr\leq(n-1+4A)\Lambda\coth\Lambda r, ~\text{ if} ~ |f|\le A,\end{align}
and
\begin{align}
\label{lapalace-r-3}\Delta^fr\leq(n-1)\Lambda\coth\Lambda r+A, ~\text{ if} ~ |\nabla f|\le A.
\end{align}
\end{lem}

As an application of Lemma  \ref{lapalace-esti-r},  Wei-Wylie prove the following weighted  volume comparison theorem.

\begin{theo}([WW])\label{volume-comparison}
Let $(M,g)$ be   an $n$-dimensional   complete  Riemannian manifold which satisfies (\ref{ricci-condition-1}).
Then for any $0<r\le R$,
\begin{align}\label{volume-estimate-1}
\frac{{\rm vol}^f(B_p(r))}{{\rm vol}^f(B_p(R))}\geq\frac{{\rm vol}_{\Lambda}^{n+4A}(B(r))}{{\rm vol}_{\Lambda}^{n+4A}(B(R))},
~\text{ if }~|f|\le A,\end{align}
and
\begin{align}\label{volume-estimate-2}\frac{{\rm vol}^f(B_p(r))}{{\rm vol}^f(B_p(R))}\geq e^{-AR}\frac{{\rm vol}_{\Lambda}^{n}(B(r))}{{\rm vol}_{\Lambda}^{n}(B(R))},
\text{ if}~ |\nabla f|\le A,
\end{align}
where ${\rm vol}_{\Lambda}^{n}(B(r))$ denotes the volume of geodesic ball $B(r)$ with radius $r$ in $n$-dimensional space form
with constant curvature $-\Lambda$.
\end{theo}

  Wei-Wylie's proof of Theorem \ref{volume-comparison} depends on a monotonic formula for  the weighted volume  form as follows.

 By choosing  a polar coordinate with the origin at $p$,  we write
\begin{align}
e^{-f}d\text{v}= A^f(s,\theta)ds\wedge d\theta.\notag\end{align}
Then
\begin{align}
 \frac{d}{ds}A^f(s,\theta)= A^f(s,\theta) \Delta^fr. \notag
 \end{align}
In case that  $|\nabla f|\le A$,  it follows  from (\ref{lapalace-r-3}),
\begin{align}\label{mono-1}
 \frac{d}{ds}A^f(s,\theta)\le A^f(s,\theta) l_{\Lambda,A}(r),
 \end{align}
 where  $ l_{\Lambda,A}(r)=(n-1)\Lambda\coth\Lambda r+A.$
  Thus if we put
   \begin{align}\label{L-function}
    L_{\Lambda,A}(r)=e^{Ar}(\frac{\sinh\Lambda r}{\Lambda})^{n-1},
    \end{align}
     which is a  solution of   equation,
\begin{align}\label{L-equation}
\frac{L_{\Lambda,A}'}{L_{\Lambda,A}}=l_{\Lambda,A}, ~\frac{L_{\Lambda,A}(r)}{ r^{n-1}}\to 1~\text{as}~  r\to 0,
 \end{align}
 (\ref{mono-1}) is  equivalent to  the following monotonic formula,
\begin{align}\label{mono-formula}
 \frac{A^f(b,\theta)}{A^f(a,\theta)}\leq \frac{L_{\Lambda,A}(b)}{L_{\Lambda,A}(a)},~\forall~ b\ge a.\end{align}
By a simple computation,
 we get (\ref{volume-estimate-2}) from (\ref{mono-formula}). Similarly, we can prove (\ref{volume-estimate-1}).

Another application of  Lemma  \ref{lapalace-esti-r} is  the following  weighted  Poincar\'{e} inequality.

 \begin{lem}\label{Poincare1}
Let  $(M,g)$  be a complete Riemannian manifold which satisfies
\begin{align}\label{curvature-uasual-condition} {\rm Ric}_{g}^f\ge -(n-1)\Lambda^2 g,~\text{and}~|\nabla f|\le A.
\end{align}
   Let  $A_p(a,b)=B_p(b)\setminus \overline{B_p(a)}$ be an annulus in $M$.
Then for  any  Liptischtz  function  $h$  in $A_p(a,b)$ with   $h|_{\partial A_p(a,b)}=0$,
it holds
  \begin{align}\label{poincare-inequality}
  \int_{A_p(a,b)}h^2e^{-f}d{\rm v} \leq c(a,b,A,\Lambda)\int_{A_p(a,b)}|\nabla h|^2e^{-f}d{\rm v}
  \end{align}

\end{lem}

\begin{proof} By  (\ref{lapalace-r-3}),  it is easy to see that
\begin{align}
\Delta^f r^{-k} &\geq -kr^{-k-1}l_{\Lambda,A}(r)+k(k+1)r^{-k-2}\notag\\
&=kr^{-k-1}(-l_{\Lambda,A}(r)+\frac{k+1}{r}),
\end{align}
where $k$ is a positive real number. Putting $\frac{k+1}{b}=l_{\Lambda,A}(a)+1$, we  have
$$\Delta^f r^{-k}\geq c(a,b,\Lambda,A)>0.$$
 Thus  for $h$ with zero boundary value, we get
\begin{align}
\nonumber c(a,b,\Lambda,A)\int_{A_p(a,b)}h^2e^{-f}d{\rm v}&\leq \int_{A_p(a,b)}(\Delta^f r^{-k}) h^2e^{-f}d{\rm v}&\\
\nonumber &=-2\int_{A_p(a,b)}h\langle \nabla h,\nabla(r^{-k})\rangle e^{-f}d{\rm v}&\\
\nonumber &\leq 2k\int_{A_p(a,b)}h|\nabla h|e^{-f}d{\rm v}&\\
\nonumber &\leq 2k (\int_{A_p(a,b)}h^2e^{-f}d{\rm v})^{\frac{1}{2}}(\int_{A_p(a,b)}|\nabla h|^2e^{-f}d{\rm v})^{\frac{1}{2}}.
\end{align}
Hence,  (\ref{poincare-inequality}) follows from the above immediately.
\end{proof}

For  the $f$-Lapalace operator,    we  have the following  Bochner-typed identity,
 \begin{align}\label{bochner-inequ}
&\frac{1}{2}\Delta^{f}|\nabla u|^2\notag\\
&=|\text{hess }u|^2+\langle\nabla u,\nabla\Delta^f u\rangle+{\rm Ric}_g^f(\nabla u,\nabla u), ~\forall~u\in C^\infty(M).
\end{align}
By  (\ref{bochner-inequ}) and Lemma \ref{lapalace-esti-r}, we derive the following Li-Yau typed gradient estimate for $f$- harmonic functions on $(M,g)$.

 \begin{prop}\label{gradient-esti} Let   $(M,g)$  be a complete Riemannian manifold which satisfies (\ref{curvature-uasual-condition}).
 Let  $u>0$ be a $f$-harmonic function defined on the  unit distance ball $B_p(1)\subset (M,g)$,  i.e.
 $$\triangle^f u=0, ~ \text{in }~B_p(1).$$
Then
 \begin{align}
 |\nabla u|^2\leq (C_1\Lambda+C_2A^2+C_3)u^2, ~\text{in}~ B_p(1/2),
 \end{align}
   where  the constants $C_i$  $(1\leq i\leq 3)$ depend only on $n$.
\end{prop}

\begin{proof} The proof is standard as  in the case  $f=0$ for a harmonic function (cf. [SY]).
 We let $v=\ln u$.
 Then
\begin{align}
\Delta^f v&=\Delta v-\langle\nabla f,\nabla v\rangle=\nabla(\frac{\nabla u}{u})-\langle\nabla f, \frac{\nabla u}{u}\rangle\notag\\
&=\frac{\Delta u}{u}-\frac{|\nabla u|^2}{u^2}-\langle\nabla f, \frac{\nabla u}{u}\rangle=\frac{|\nabla u|^2}{u^2}.
\end{align}
Note that
$$|\text{hess } v|^2\geq\frac{|\Delta v|^2}{n}$$
and
$$|\Delta v|^2\geq\frac{|\Delta^f v|^2}{2}-C_1 A^2Q,$$
where $Q=|\nabla v|^2$. Thus applying (\ref{bochner-inequ}) to $v$,  we get
\begin{align}\label{ineq-Q}
\frac{1}{2}\Delta^f Q\geq\frac{Q^2}{2n}-\frac{1}{n}C_1 A^2Q+\langle\nabla v,\nabla Q\rangle-\Lambda^2Q .
\end{align}

Choose a decreasing cut-off function $\eta(t)$  on  $t\in [0,\infty]$ such that
 \begin{align}&\eta(t)=1 ~\text{if}~  t\le \frac{1}{2}; \phi=0~\text{ if}~ t\ge 1;\notag\\
   &-C_2\eta^{\frac{1}{2}}\le \eta^{\prime},~\text{if}~t\ge \frac{1}{2};\notag\\
   &|\eta^{\prime\prime}|\leq C_2.\notag
   \end{align}
Then  if let  $\phi=\eta(r(\cdot,p))$,
$$|\nabla\phi|^2\phi^{-1}\le C_2^2,$$
 and by Lemma  \ref{lapalace-esti-r},
\begin{align}\label{lapalace-phi}\Delta^f\phi=\Delta^f\eta(r)\ge - C_3(A+\Lambda).
\end{align}
Hence, by  (\ref{ineq-Q}),  we obtain
\begin{align}\label{ineq-phiQ}
\Delta^f(\phi Q)&=\Delta(\phi Q)+\langle\nabla f,\nabla(\phi Q)\rangle\notag\\
 &= \phi\Delta^f Q+Q\Delta^f\phi+ 2\langle\nabla Q,\nabla \phi\rangle\notag\\
 &\ge \phi(\frac{Q^2}{n}-(\frac{2}{n}C_1 A^2+2\Lambda^2)Q)-C_3(A+\Lambda)Q\notag\\
 &+2\langle\nabla v,\nabla Q\rangle+2\langle\nabla Q,\nabla \phi\rangle.
 \end{align}

Suppose that $(Q\phi)(q)=\max_M\{Q\phi\}$ for some $q\in M$. Then at this point, it holds
 $\nabla(Q\phi)=0$.
 It follows that
 $$\nabla Q=-\frac{Q\nabla \phi}{\phi},$$
and
\begin{align}|\langle\nabla Q,\nabla\phi\rangle|=\frac{Q}{\phi}|\nabla\phi|^2\leq C_2^2Q. \notag
\end{align}
Also
\begin{align}
|\langle\nabla Q,\nabla v\rangle|\leq Q^{\frac{3}{2}}\frac{|\nabla\phi|}{\phi}\le
C_2Q^{\frac{3}{2}}\phi^{-\frac{1}{2}}.\notag\end{align}
Therefore, by applying  the maximum principle to $\phi Q$ at the point $q$,  we get from (\ref{ineq-phiQ}),
\begin{align}
0&\geq\phi(\frac{Q^2}{n}-\frac{2}{n}C_1A^2Q- 2C_2Q^{\frac{3}{2}}\phi^{-\frac{1}{2}})\notag\\
 &-C_3(\Lambda Q+A)-2C_2Q.\notag
\end{align}
As a consequence,  we derive
$$ \phi Q\le (\phi Q)(q)\leq C_4\Lambda+C_5A^2+C_6, ~\text{in}~ B_p(1).$$
This proves  the proposition.\end{proof}

As an application of  Proposition \ref{gradient-esti}, we are able to construct a cut-off function with  
 bounded gradient and $f$-Lapalace.  Such a function will be used in the next section.

\begin{lem}\label{cut-off}
Under the condition (\ref{curvature-uasual-condition})  in Lemma \ref{Poincare1}, there exists a cut-off  function $\phi$ supported in $B_p(2)$ such that
i)  $\phi\equiv1$,  in  $B_p(1)$;  ii)
\begin{align}
 |\nabla \phi|, |\Delta^f\phi| \le  C(n,\Lambda, A).\notag
\end{align}
\end{lem}

\begin{proof} We will use an argument from Theorem  6.33 in [CC1].
First we consider a solution  of ODE,
\begin{align}
G''+G'l_{\Lambda,A}=1, ~\text{on}~ [1,2],
\end{align}
with $G(1)=a$  and  $G(2)=0$.   It is easy to see that there is a number $a=a(n,\Lambda, A)$
such tha $G'<0$.
Then  by (\ref{lapalace-r-3}),  we have
$$\Delta^f G(d(p,\cdot))\geq1.$$
Let $w$ be a  solution of equation,
\begin{align}
\Delta^f w=\frac{1}{a}, ~\text{in}~B_p(2)\setminus \overline{B_p(1)},\notag \end{align}
with $w=1$  on  $\partial B_p(1)$ and  $ w=0$  on  $\partial B_p(2)$.
 Thus  by the maximum principle,  we get
  $$w\geq \frac{G(d(.,p))}{a}.$$

Secondly,  we choose another  function $H$  with  $H'>0$ which is a solution of ODE,
\begin{align}
H''+H'l_{\Lambda,A}=1, ~ \text{ on } [0,\infty),
\end{align}
with $H(0)=0$.  Then  by (\ref{lapalace-r-3}),   we have
$$\Delta^f H(d(x,\cdot))\leq 1,\text{ for any fixed point} ~ x. $$
Thus by  the maximum principle,  we get
$$w(y)-\frac{H(d(x,y))}{a}\leq max\{1-\frac{H(d(x,p)-1)}{a},0\}$$
for any $y$ in the annulus  $A_p(1,2)=B_p(2)\setminus \overline{B_p(1)}$.
It follows
$$w(x)\leq max\{1-\frac{H(d(x,p)-1)}{a}, 0\},~\forall~ x\in A_p(1,2).$$

 Now  we choose a number  $\eta(n,\Lambda,A)$ such that $\frac{G(1+\eta)}{a}>1-\frac{H(1-\eta)}{a}$ and we  define a function $\psi(x)$ on $[0,1]$  with bounded derivative up to second order,  which satisfies
\begin{align}
\psi(x)=1,  \text{ if } x\geq \frac{G(1+\eta)}{a}\notag
\end{align}
 and
 \begin{align}
  \psi(x)=0, \text{ if } x\leq max\{1-\frac{H(1-\eta)}{a}, 0\}.\notag
\end{align}
 It is clear that  $\phi=\psi\circ w$ is constant near the boundary of $A_p(1,2)$.  So we  can extend  $\phi$  inside $B_p(1)$ by setting $\phi=1$.  By Proposition  \ref{gradient-esti},  one sees that  $|\nabla\phi|$ is bounded by a constant $C(n,\Lambda, A)$
 in $B_2(p)$.    Since
 $$\Delta^f \phi=\psi''|\nabla w|^2+\psi'\Delta^f w, $$
 we also derive  that $|\Delta^f\phi| \le C(n,\Lambda, A)$.
\end{proof}

 \vskip3mm

\section{$L^2$-Integral estimates for  Hessians of functions}

In this section, we  establish  various integral comparisons   of gradient and Hessian between    appropriate $f$-harmonic functions and   coordinate
functions  or  distance functions.  We start with a basic lemma about  a distance function  along  a long approximate line in a  manifold.

\begin{lem}\label{coordinate-lemma1}
Let $(M,g)$ be a complete Riemannian manifold which satisfies
\begin{align}\label{flat-ricci-condition} {\rm Ric}_g^f\geq-\frac{n-1}{R^2} g~\text{and}~|f|\leq A.\end{align}
  Suppose that there are three points $p, q^+, q^-$  in $M$ which satisfy
\begin{align}
d(p,q^+)+d(p,q^-)-d(q^+,q^-)<\epsilon\end{align}
and
\begin{align}d(p,q^+),d(p,q^-)> R.
\end{align}
Then  for  any $q \in B_p(1)$,  the following  holds,
$$E(q):=d(q,q^+)+d(q,q^-)-d(q^+,q^-) <\Psi(\epsilon,\frac{1}{R};A, n), $$
where the quantity  $\Psi(\epsilon, \frac{1}{R}; A, n)$ means that it goes to zero as   $\epsilon, \frac{1}{R}$  go to zero  while $A, n$ are  fixed.
\end{lem}

\begin{proof} Let
 \begin{align}\label{lap-com}
\tilde l(s)=(n-1+4A)\frac{1}{R}\coth\frac{s}{R}.
\end{align}
For given $t>0$, we construct a function $G=G_t(s)$ on $[0,t]$ which  satisfies  the ODE,
\begin{align}\label{comparison-function}
 G''+\tilde l(s)G'=1, G'(s)<0,
\end{align}
with  $G(0)=+\infty$ and  $G(t)=0$.
  Then $G(s)\sim s^{2-n-4A}$ $(s\rightarrow 0)$.
   Furthermore, by  (\ref{lapalace-r-1}) in Lemma \ref{lapalace-esti-r},  we  have
\begin{align}
\Delta^f G(d(x,.))=G'\Delta^fd(x,.)+G''\geq G''+G'\tilde
l(s)=1.
\end{align}

By  Lemma \ref{lapalace-esti-r},
\begin{align}
\Delta^f E(q)\leq \frac{10(n-1+A)}{R}:=b.
\end{align}
We claim:
For any $0<c<1$ ,
$$E(q)\leq 2c+bG_1(c)+\epsilon,~\text{ if}~ bG_1(c)>\epsilon.$$

  Suppose that the claim is not true.  Then there exists  point $q_0\in  B_p(1)$ such that
 for some $c$,
  $$E(q_0)>2c+bG_1(c)+\epsilon.$$
  We  consider $u(x)=bG_1(d(q_0,x))-E(x)$ in the annulus $A_{q_0}(c,1)$.  Clearly,
  $$\Delta^f u\geq 0.$$
   Note that we may assume that  $p \in A_{q_0}(c,1)$.  Otherwise $d(q_0,p)<c$ and
    $E(q_0)\leq E(p)+2c$,  so the claim is true and the  proof is complete.  On the other hand,
    it is easy to see that  on the inner boundary $\partial B_{q_0}(c)$,
\begin{align}
\nonumber u(x)=bG_1(c)-E(x)\leq bG_1(c)-E(q_0)-2c\leq-\epsilon,
\end{align}
and on the outer boundary $\partial B_q(1)$,
\begin{align}
\nonumber u(x)=-E(x)\leq 0.
\end{align}
Thus applying  the maximum principle,  it follows that  $ u(p)\le 0$.
However,
$$
u(p)=bG_1(d(p,q_0))-E(p)\geq bG_1(c)-\epsilon,
$$
which is impossible.  Therefore,  the claim is true.

By choosing  $c$  with the order $b^{\frac{1}{n-1+4A}}$ in the above claim,
we  prove Lemma \ref{coordinate-lemma1}.

\end{proof}

Let $b^+(x)=d(q^+,x)- d(q^+,p)$ and let  $h^+$  be a $f$-harmonic  function which satisfies
$$\triangle^f h^+=0,~ \text{in} ~B_p(1),$$
 with $h^+=b^+$ on $\partial B_p(1)$. Then

 \begin{lem}\label{harmonic-estimate} Under the conditions in Lemma \ref{coordinate-lemma1} with $|\nabla f|\le A$,
 we have
\begin{align}\label{c0-h}
\|h^+-b^+\|_{C^0(B_p(1))}<\Psi(1/R,\epsilon; A),\end{align}
\begin{align}\label{h-gradient}\frac{1}{{\rm vol}(B_p(1))}\int_{B_p(1)}|\nabla h^+-\nabla b^+|^2
e^{-f}\text{d{\rm v}} <\Psi(1/R,\epsilon; A),\end{align}
\begin{align}\label{hessian-integral}\frac{1}{{\rm vol}(B_p(\frac{1}{2}))}
\int_{B_p(\frac{1}{2})} |\text{hess }h^+|^2e^{-f}\text{d{\rm v}}
<\Psi(1/R,\epsilon; A).
\end{align}
\end{lem}

\begin{proof}
Choose a point   $q$ in $\partial B_p(2)$ and let $g=\varphi(d(q,\cdot))$, where  $\varphi$ is a solution of
(\ref{comparison-function}) restricted on the interval $[1,3]$.
Then
\begin{align}
\Delta^fg=\varphi'\Delta^fr+\varphi''\geq\varphi' \tilde l+\varphi'' =1,~\text{in}~ B_p(1).
\end{align}
 It follows that
 \begin{align}
\nonumber \Delta^f (h^+-b^++\Psi(1/R,\epsilon; A)g) > 0,~\text{in}~ B_p(1).
\end{align}
Thus by  the maximum principle,  we get
$$h^+-b^+ < \Psi(1/R,\epsilon; A).$$
On the other hand, we have
$$\Delta^f ( -b^--h^++\Psi(1/R,\epsilon; A)g)  > 0,~\text{in}~ B_p(1),$$
where $b^-=d(q^-,x)- d(p,q^{-})$.  Since  $b^++b^-$ is small as long as $1/R$ and $\epsilon$ are small by Lemma \ref{coordinate-lemma1}, by  the maximum principle, we also
 get
 $$h^+-b^+ > -(b^++b^-)-\Psi(1/R,\epsilon; A) > -\Psi(1/R,\epsilon; A).$$

 For the second estimate (\ref{h-gradient}),  we see
\begin{align}
& \int_{B_p(1)}|\nabla h^+-\nabla b^+|^2e^{-f}d\text{v} \notag\\
&=\int_{B_p(1)}(h^+-b^+)(\triangle^f b^+-\triangle^fh^+)e^{-f}d\text{v}\notag\\
 & <\Psi(1/R,\epsilon; A)\int_{B_p(1)}|\triangle^f b^+|e^{-f}d\text{v}.\notag
\end{align}
and
\begin{align}
& \int_{B_p(1)}|\triangle^f b^+|e^{-f}d\text{v}\notag\\
&\leq|\int_{B_p(1)}\triangle_f b^+e^{-f}d\text{v}|+2e^{A} \text{sup}_ {B_p(1)}(\triangle^f b^+){\rm vol}(B_p(1))\notag\\
 &\leq  e^{A}{\rm vol}({\partial B_p(1)})+C(A){\rm vol}(B_p(1))\notag\\
&\leq C'(A){\rm vol}(B_p(1)).\notag
\end{align}
Here we used (\ref{mono-formula}) at the last inequality.  Then (\ref{h-gradient}) follows.

To get  (\ref{hessian-integral}), we  choose a cut-off function $\varphi$ supported in $B_p(1)$  as constructed in  Lemma \ref{cut-off}.
Since
\begin{align}
&\nonumber \Delta^f(|\nabla h^+|^2-|\nabla b^+|^2)\notag\\
&=|\text{hess }h^+|^2+{\rm Ric}_g^f(\nabla h^+,\nabla h^+), \notag
\end{align}
  multiplying both sides of the above  by $\varphi e^{-f}d\text{v}$ and using integration by parts,
 we  get
 \begin{align}&\int_{B_p(1)} \varphi|\text{Hess }h^+|^2 e^{-f}d\text{v}\notag\\
 &\le \int_{B_p(1)} \Delta^f \varphi (|\nabla h^+|^2-|\nabla b^+|^2) e^{-f}d\text{v}+\frac{n-1}{R^2}\int_{B_p(1)}
 \varphi|\nabla h^+|^2 e^{-f}d\text{v}\notag.
 \end{align}
Note that  $|\nabla h^+|$ is locally bounded by Proposition  \ref{gradient-esti},
  we derive (\ref{hessian-integral}) from (\ref{h-gradient}) immediately.
\end{proof}

Next, we  construct an approximate function  to compare  the square  of a distance function with  asymptotic  integral  gradient and Hessian estimates.  Such estimates  are crucial  in the proof of metric-cone theorem in Section 4.

Let  $q\in M$ and  $h$ be a solution of the following equation,
\begin{align}\label{f-harmonic-radial}
\Delta^fh=n, ~\text{in}~ B_q(b)\setminus \overline{B_q(a)}, ~h|\partial B_q(b)=\frac{b^2}{2}~\text{ and}~   h|\partial B_q(a)=\frac{a^2}{2}.
\end{align}
   Let $p=\frac{r(q,\cdot)^2}{2}$. Then

 \begin{lem}\label{harmonic-estimate-annual-1}
 Let  $(M,g)$  be a complete Riemannian manifold which satisfies
  $${\rm Ric}^f_g\geq-(n-1)\epsilon^2\Lambda^2g ~\text{and}~|\nabla f|\leq\epsilon A.$$
Let     $a<b$.  Suppose that
    \begin{align}\label{cone-volume-condition-2}
 \frac{{\rm vol}^f(\partial B_q(b))}{{\rm vol}^f(\partial B_q(a))}\geq(1-\omega)\frac{L_{\epsilon\Lambda,\epsilon A}(b)}{L_{\epsilon\Lambda,\epsilon A}(a)}
 \end{align}
for some $\omega>0$, where $L_{\epsilon\Lambda,\epsilon A}(r)$ is the function defined by (\ref{L-function}) with respect to constants  $\epsilon\Lambda$ and $\epsilon A.$
Then
 \begin{align}\label{gradient-annual}\frac{1}{{\rm vol}(A_q(a,b))}
 \int_{A_q(a,b)}|\nabla p-\nabla h|^2  e^{-f}d\text{v} < \Psi(\omega,\epsilon;\Lambda, A,a,b).
 \end{align}
Moreover,
\begin{align}\label{h-p-estimate}
 \|h-p\|_{C^0(A_q(a',b'))}< \Psi(\omega,\epsilon;\Lambda, A,a,b,a',b'),
 \end{align}
 where   $a<a'<b'<b$.

 \end{lem}

 \begin{proof}  Since
 $$\Delta^f r\leq(n-1)\epsilon \Lambda\coth(\epsilon \Lambda r)+\epsilon A=l_{\epsilon\Lambda,\epsilon A},$$
 we have
 \begin{align}\label{lapalace-p}
  \Delta^fp=p''+p' \Delta^f r  < n+\Psi(\epsilon;\Lambda, A, a, b), ~\text{ in }A(a,b).
   \end{align}
   Thus we get
\begin{align}\label{lemma-2-5-1}
 \frac{1}{{\rm vol}(A_q(a, b))}\int_{A_q(a,b)}\Delta^fpe^{-f} d\text{v}<   e^{-f(0)}( n+\Psi(\epsilon;\Lambda, A, a, b)).
\end{align}
 On the other hand, by   the monotonicity formula (\ref{mono-formula}), we have
\begin{align}
\nonumber \frac{\int_a^bL_{\epsilon\Lambda,\epsilon A}(s)ds}{L_{\epsilon\Lambda,\epsilon A}(b)}{\rm vol}^f(\partial B_q(b))\leq{\rm vol}^f(A_q(a, b))\leq \frac{\int_a^bL_{\epsilon\Lambda,\epsilon A}(s)ds}{L_{\epsilon\Lambda,\epsilon A}(a)}{\rm vol}^f(\partial B_q(a)).
\end{align}
It follows by (\ref{cone-volume-condition-2}),
$${\rm vol}^f(A_q(a, b))\leq  (1-\omega)^{-1}\frac{\int_a^bL_{\epsilon\Lambda,\epsilon A}(s)ds}{L_{\epsilon\Lambda,\epsilon A}(b)}{\rm vol}^f(\partial B_q(b)).$$
Since
\begin{align}
 \nonumber \int_{A_q(a,b)}\Delta^fpe^{-f} d\text{v}=b{\rm vol}^f(\partial B_q(b)) -a{\rm vol}^f(\partial B_q(a)),
 \end{align}
we get
\begin{align}
 \nonumber &\frac{1}{{\rm vol}^f(A_q(a, b))} \int_{A(a,b)}\Delta^fpe^{-f} d\text{v}\\
 &\ge
  (1-\omega)\frac{L_{\epsilon\Lambda,\epsilon A}(b)} {\int_a^bL_{\epsilon\Lambda,\epsilon A}(s)ds} ( b-
a\frac{{\rm vol}^f(\partial B_q(a))}{ {\rm vol}^f(\partial B_q(b))}). \notag
 \end{align}
Observe that
${\rm vol}^f$ is close to $e^{-f(0)}{\rm vol}$ and $\frac{L_{\epsilon\Lambda,\epsilon A}(s)}{s^{n-1}}$ is close to a constant as $\epsilon$ is small.   Hence we derive immediately,
\begin{align}\label{lemma-2-5-2}
 \frac{1}{{\rm vol}(A_q(a, b))}\int_{A_q(a,b)}\Delta^fpe^{-f} d\text{v}>  e^{-f(0)}   (n+\Psi(\omega,\epsilon;\Lambda, A, a, b)).
\end{align}

By (\ref{lemma-2-5-1}) and (\ref{lemma-2-5-2}), we have
$$|\int_{A_q(a,b)}(\Delta^fp-n) e^{-f} d\text{v}|<  {\rm vol}(A_q(a, b)) \Psi(\omega,\epsilon;\Lambda, A, a, b).$$
Then  one can follow the argument  for the estimate (\ref{h-gradient}) in Lemma \ref{harmonic-estimate} to obtain (\ref{gradient-annual}).

Applying  Lemma \ref{Poincare1} to the function $p-h$ together with   the estimate (\ref{gradient-annual}),  we see that
\begin{align}
\nonumber \frac{1}{{\rm vol}^f(A_q(a,b))}\int_{A_q(a,b)}|p-h|^2  e^{-f}d\text{v} < \Psi(\omega,\epsilon;\Lambda, A,a,b).
\end{align}
Then  for any point $x\in A_q(a',b')$,   by (\ref{volume-estimate-2}),
there is a point $y\in B_x(\eta)$  such that
\begin{align}
|p(y)-h(y)|^2&\leq \frac{{\rm vol}^f(A_q(a,b))}{{\rm vol}^f(B_x(\eta))}  \frac{1}{{\rm vol}^f(A_q(a,b))}\int_{A_q(a,b)}|p-h|^2  e^{-f}d\text{v}\notag\\
&< \frac{C(\Lambda, A, b)}{\eta^n}\Psi(\omega,\epsilon;\Lambda, A,a,b).\notag\end{align}
On the other hand,  by  Proposition \ref{gradient-esti}, we have
\begin{align} |(p(x)-h(x))-(p(y)-h(y))| &\leq (\|\nabla h\|_{C^0(A_q(a'-\eta,b'+\eta))}+1)\text{dist}(x,y)\notag\\
&\leq C(\Lambda,A,a,b,a'-\eta,b'+\eta)\eta.\notag\end{align}
Thus  we derive
\begin{align}
&|p(x)-h(x)|\notag\\
&< \frac{C(\Lambda, A,b)}{\eta^n}\Psi(\omega,\epsilon;\Lambda, A,a,b)+C(\Lambda,A,a,b,a'-\eta,b'+\eta)\eta.\notag
\end{align}
Choosing $\eta=\Psi^{\frac{1}{n+1}}$, we  prove  (\ref{h-p-estimate}).
\end{proof}

Furthermore,  we have

 \begin{lem}\label{harmonic-estimate-annual-2} Under the condition in Lemma \ref{harmonic-estimate-annual-1}, it holds
 \begin{align}&\frac{1}{{\rm vol}(A_q(a_2,b_2))}\int_{A_q(a_2,b_2)}|{\rm hess }h-g|^2e^{-f} d\text{v}\notag\\
 &< \Psi(\omega,\epsilon;\Lambda, A,a_1,b_1,a_2,b_2,a,b),
 \end{align}
  where  $a<a_1<a_2<b_2<b_1<b$.
   \end{lem}

 \begin{proof} First observe that
 \begin{align}
 \nonumber &\frac{1}{{\rm vol}(A_q(a,b))}\int_{A_q(a,b)}|\text{hess }h-g|^2e^{-f} d\text{v}\\
 \nonumber &=\frac{1}{{\rm vol}(A_q(a,b))}\int_{A_q(a,b)}|{\rm hess}h|^2e^{-f} d\text{v}
  +\frac{1}{{\rm vol}(A_q(a,b))}\int_{A_q(a,b)}(n-2\Delta h)e^{-f}\text{v}.
 \end{align}
Let  $\varphi$ be a cut-off function  of $A_q(a,b)$ with  support in $A_q(a_1,b_1)$  as constructed  in Lemma \ref{cut-off} which satisfies,
\begin{align}
& 1)   ~\varphi\equiv 1,~\text{ in}~ A_q(a_2,b_2);\notag\\
&  2) ~|\nabla\varphi|,  |\triangle^f\varphi|~\text{  is bounded in}~ A_q(a,b).\notag
\end{align}
 Then
 \begin{align}\label{two-parts}
 \nonumber &\frac{1}{{\rm vol}(A_q(a,b))}\int_{A_q(a,b)}\varphi|{\rm hess}h-g|^2e^{-f} d\text{v}\\
 \nonumber &=\frac{1}{{\rm vol}(A_q(a,b))}\int_{A_q(a,b)}\varphi|\text{hess }h|^2e^{-f} d\text{v}\\
 &+\frac{1}{{\rm vol}(A_q(a,b))}\int_{A_q(a,b)}\varphi(n-2\Delta h)e^{-f} d\text{v}.
 \end{align}

 By  the  Bochner formula (\ref{bochner-inequ}), we have
  \begin{align}
 \nonumber  &\frac{1}{{\rm vol}(A_q(a,b))}\int_{A_q(a,b)}\varphi|\text{hess }h|^2e^{-f} d\text{v}\\
 \nonumber  &<\frac{1}{2{\rm vol}(A_q(a,b))}\int_{A_q(a,b)}\varphi\Delta^f|\nabla h|^2e^{-f} d\text{v}+ \Psi(\epsilon;\Lambda,A,a_1,b_1,a_2,b_2,a,b).
 \end{align}
 It follows by Lemma \ref{harmonic-estimate-annual-1},
 \begin{align}
 \nonumber  &\frac{1}{{\rm vol}(A_q(a,b))}\int_{A_q(a,b)}\varphi|\text{hess }h|^2e^{-f} d\text{v}\\
 \nonumber &<\frac{1}{2{\rm vol}(A_q(a,b))}\int_{A_q(a,b)}\varphi\Delta^f|\nabla p|^2e^{-f} d\text{v}+ \Psi(\epsilon,\omega;\Lambda,A,a_1,b_1,a_2,b_2,a,b) \\
 \nonumber &=\frac{1}{{\rm vol}(A_q(a,b))}\int_{A_q(a,b)}\varphi\Delta^fp e^{-f} d\text{v}+ \Psi(\epsilon,\omega;\Lambda,A,a_1,b_1,a_2,b_2,a,b).
 \end{align}
On the other hand,
 \begin{align}
 \nonumber &\frac{1}{{\rm vol}(A_q(a,b))}\int_{A_q(a,b)}\varphi(n-2\Delta h)e^{-f} d\text{v}&\\
\nonumber & =\frac{1}{{\rm vol}(A_q(a,b))}\int_{A_q(a,b)}\varphi(-n-2\langle\nabla f,\nabla h\rangle)e^{-f} d\text{v}&\\
 \nonumber &=\frac{1}{{\rm vol}(A_q(a,b))}\int_{A_q(a,b)}-n\varphi e^{-f} d\text{v}+\Psi(\epsilon,\omega;\Lambda,A,a_1,b_1,a,b).&
 \end{align}
 Hence we derive from (\ref{two-parts}),
 \begin{align}
 \nonumber &\frac{1}{{\rm vol}(A_q(a_2,b_2))}\int_{A_q(a_2,b_2)}|\text{hess }h-g|^2e^{-f} d\text{v}\\
 \nonumber & \leq\frac{1}{{\rm vol}(A_q(a,b))}\int_{A_q(a,b)}\varphi|\text{hess }h-g|^2e^{-f} d\text{v}\\
 \nonumber &<\frac{1}{{\rm vol}(A_q(a,b))}\int_{A_q(a,b)}\varphi(\Delta^fp-n)e^{-f} d\text{v}+ \Psi(\epsilon,\omega;\Lambda,A,a_1,b_1,a_2,b_2,a,b)
 \\
 \nonumber &<  \Psi(\epsilon,\omega;\Lambda,A,a_1,b_1,a_2,b_2,a,b).
 \end{align}
Here we used  (\ref{lapalace-p}) at last inequality.
\end{proof}

 \section{A splitting theorem}

In this section, we  prove  the splitting theorem of Cheeger-Colding   in  the Bakry-\'{E}mery geometry [CC1].
Recall that $\gamma(t)$ $(t\in (-\infty,
\infty))$ is  a line in a metric space $Y$ if
$$\text{dist}(\gamma(t_1),\gamma(t_2))=|t_1-t_2|, ~\forall~t_1,t_2\in (-\infty,\infty).$$

\begin{theo}\label{splitting-theorem}
 Let $(M_i,g_i; p_i)$ be a sequence of Riemannian manifolds which satisfy
  $${\rm Ric}_{M_i,g_i}^{f_i}\geq{-\epsilon_i^2} g_i,~|f_i|,~|\nabla f_i|\leq A.$$
  Let   $(Y;y)$ be  a  limit metric space of   $(M_i,g_i; p_i)$  in the pointed Gromov-Hausdorff  topology  as $\epsilon_i\rightarrow 0$.
  Suppose that  $Y$ contains a line passing $y$.  Then $Y=\mathbb{R}\times X$ for some metric space $X$.
\end{theo}

We will follow the argument in [CC1] to prove  Theorem \ref{splitting-theorem}.
 The proof depends on the  following triangle lemma  in terms  of small     integral    Hessian  of appropriate function.

\begin{lem}\label{cheeger-lemma}
Let $x,y,z$ be  three points in a complete  Riemannian manifold $M$.  Let  $\gamma(s)$ $(s\in [0,a], ~a=d(x,y))$  be
 a geodesic  curve connecting $x,y$ and
 $\gamma_s(t)$ $(s\in [0,l(s)], ~l(s)=d(z,\gamma (s)))$  a family of geodesic curves  connecting $z$ and $\gamma(s)$.
Suppose that  $h$  is  a smooth  function on $M$ which satisfies
\begin{align} &i) ~ |h(z)-h(x)|<\delta<<1;\notag\\
 & ii)~ \int_{[0,a]}|\nabla h(\gamma(s))-\gamma'(s)|<\delta<<1;
 \notag\\
&iii)~\int_{[0,a]}\int_{[0,l(s)]}|{\rm hess} ~h(\gamma_s(t))|dtds<\delta<<1.\notag
\end{align}
 Then
 \begin{align}\label{triangular-equ} |d(z,x)^2+d(x,y)^2-d(y,z)^2|<\epsilon(\delta)<<1.\end{align}
\end{lem}

\begin{proof}  The proof below comes  essentially from Lemma 9.16 in [Ch2].  First by  the condition ii), we have
\begin{align}
|h(\gamma(s))-h(\gamma(0))-s|=|\int_0^s(\langle \nabla h(\gamma(s))-\gamma'(s),\gamma'(s)\rangle|\leq \delta.\notag
\end{align}
 Then
 $$s= h(\gamma(s))-h(x)+o(1).$$
By the condition  i),  it follows
\begin{align}\frac{1}{2} d(x,y)^2&=\int_0^a sds\notag\\
&=\int_0^a (h(\gamma(s))-h(x))ds+o(1)\notag\\
&=\int_0^a (h(\gamma_s(l(s)))-h(\gamma_s(0)))ds+o(1)\notag\\
&=\int_0^{l(s)}\int_0^a \langle\nabla h(\gamma_s(t)),\gamma_s'(t)\rangle dtds+o(1).\notag
\end{align}
On the other hand,
\begin{align} &| \langle\nabla h(\gamma_s(t)),\gamma_s'(t)\rangle- \langle\nabla h(\gamma_s(l(s))),\gamma_s'(l(s))\rangle|\notag\\
&=|\int_t^{l(s)} {\rm hess}h(\gamma_s'(\tau),  \gamma_s'(\tau)) d\tau|\notag\\
&\le \int_0^{l(s)} |{\rm hess}h(\gamma_s'(t),  \gamma_s'(t))| dt.\notag
\end{align}
Hence from the condition iii),  we get
\begin{align}\label{short-distance}\frac{1}{2} d(x,y)^2&=\int_0^{l(s)}\int_0^a \langle\nabla h(\gamma_s(l(s))),\gamma_s'(l(s))\rangle dtds+o(1)\notag\\
&=\int_0^a \langle\nabla h(\gamma_s(l(s))),\gamma_s'(l(s))\rangle l(s)ds+o(1)\notag\\
&=\int_0^a \langle\nabla  h(\gamma(s)),\gamma_s'(l(s))\rangle l(s)ds+o(1).
\end{align}

Secondly,  by  the first variation formula of geodesic curve,   we see that
$$
l'(s)=\langle\gamma_s'(l(s)),\gamma'(s)\rangle.$$
Then by  the condition ii), we obtain
\begin{align}
&\int_0^a \langle\nabla  h(\gamma(s)),\gamma_s'(l(s))\rangle l(s)ds\notag\\
&=\int_0^a  l'(s) l(s)ds +o(1)\notag\\
&=\frac{1}{2} (d(y,z)^2-d(z,x)^2).\notag
\end{align}
Therefore, combining (\ref{short-distance}), we derive (\ref{triangular-equ}).

\end{proof}

In order to  get the above configuration in Lemma \ref{cheeger-lemma}, we need a  segment inequality lemma
  in terms  of  the Bakry-\'{E}mery Ricci curvature.
In the following,  we  will always assume that the manifold $(M,g)$ satisfies
\begin{align}\label{be-curvature-condition}
{\rm Ric}^f_g\geq-(n-1)\Lambda^2 g, ~|f|, |\nabla f|\leq A,
\end{align}
and   the volume form $d\text{v} $ is replaced by $d\text{v}^f=e^{-f}d\text{v}$.

\begin{lem}\label{equ-seg}
Let  $A_1, A_2$ be two  subsets of $M$ and $W$  another subset of $M$ such that
$\bigcup_{y_1\in A_1,y_2\in A_2}\gamma_{y_1y_2}\subseteq W$, where  $\gamma_{y_1y_2}$ is
 a minimal geodesic curve connecting $y_1$ and $y_2$  in $M$.  Let
$$D=sup\{d(y_1, y_2)|~y_1\in A_1,y_2\in A_2\}.$$
 Then for any  smooth  function  $e$   on $W$,  it holds
\begin{align}\label{segment-inequ}
&\int_{A_1\times A_2}\int_0^{d(y_1,y_2)}e(\gamma_{y_1,y_2}(s))ds\notag\\
&\leq c(n,\Lambda,A)D[{\rm vol}^f(A_1)+{\rm vol}^f(A_2)]\int_We,
\end{align}
where  $c(n,\Lambda,A)=sup_{s,u}\{L_{\Lambda, A}(s)/L_{\Lambda, A}(u)|~0<\frac{s}{2}\leq u\leq s\}$.
\end{lem}

\begin{proof} Note that
\begin{align}
\nonumber &\int_{A_1\times A_2}\int_0^{d(y_1,y_2)}e(\gamma_{y_1,y_2}(s))ds&\\
\nonumber &=\int_{A_1}dy_1\int_{A_2}\int_{\frac{d(y_1,y_2)}{2}}^{d(y_1,y_2)}e(\gamma_{y_1y_2}(s))dsdy_2\\
&\nonumber+\int_{A_2}dy_2\int_{A_1}\int_{\frac{d(y_1,y_2)}{2}}^{d(y_1,y_2)}e(\gamma_{y_1y_2}(s))dsdy_1.&
\end{align}
On the other hand, for a fixed $y_1\in A_1$, by using  the monotonicity  formula  (\ref{mono-formula}), we have
\begin{align}
\nonumber &\int_{A_2}\int_{\frac{d(y_1,y_2)}{2}}^{d(y_1,y_2)}e(\gamma_{y_1y_2}(s))dsdy_2\\
\nonumber&=\int_{A_2}\int_{\frac{r}{2}}^{r}e(\gamma_{y_1y_2}(s))A^f(r,\theta)drd\theta  ds&\\
\nonumber &\leq c(n,\Lambda,A)\int_{A_2}\int_{\frac{r}{2}}^{r}e(\gamma_{y_1y_2}(s))A^f(s,\theta)drd\theta ds\\
\nonumber &\leq c(n,\Lambda,A)D\int_We.&
\end{align}
Similarly,
\begin{align}
\nonumber &\int_{A_1}\int_{\frac{d(y_1,y_2)}{2}}^{d(y_1,y_2)}e(\gamma_{y_1y_2}(s))dsdy_1\\
\nonumber &\leq c(n,\Lambda,A)D\int_We.&
\end{align}
Then (\ref{segment-inequ}) follows from the above two inequalities.

\end{proof}

Using the same argument above,  we can prove

\begin{lem}\label{equ-rad}
Given two points $q^-,q$ with $d(q,q^-)\ge 10$ and a  smooth function $e$ with support in $B_p(1)$,  then for any  $B_q(r)\subset  B_p(1)$  the following inequality holds,
\begin{align}
\int_{B_q(r)}dy\int_0^{d(q^-,y)}e(\gamma_{q^-y}(s))ds\leq c(\Lambda, A)\int_{B_p(1)}e(y)dy.
\end{align}
\end{lem}

Combining  Lemma \ref{triangular-equ} and Lemma \ref{equ-rad}, we get another  segment inequality lemma as follows.

\begin{lem}\label{approxi-1} Let $ b^+(q)=d(q,q^+)-d(p,q^+)$ for any $q$ with $d(q,q^+)\geq 10$.
Let   $h^+$   be a smooth function which satisfies
\begin{align}
\nonumber \int_{B_p(1)}|\nabla h^+-\nabla b^+|\leq \epsilon {\rm vol}^f(B_p(1))\end{align}
and
 \begin{align}\nonumber \int_{B_p(1)}|{\rm hess }~h^+|\leq \epsilon {\rm vol}^f(B_p(1)).
\end{align}
We assume that   Lemma \ref{equ-seg} and Lemma \ref{equ-rad} are true.
Then for any two points $q,q' \in B_p(\frac{1}{8})$ and any  small  number $\eta>0$, there  exist  $y^*,z^*$ with
$d(y^*,q)<\eta, d(z^*,q')<\eta$,
 and  a minimal geodesic line  $\gamma(t)$ $(0\leq t\leq l(y^*))$   from $y^*$ to $q^{-}$ with
 $\gamma(0)=y^*,\gamma(l(y^*))\in \partial B_p(\frac{1}{8})$
   such that the following is true:
 \begin{align}\label{lemma-3.5-1}
 \int_0^{l(y^*)}|\nabla h^+(s)-\gamma'(s)|ds\leq\epsilon\frac{{\rm vol}^f(B_q(2))}{{\rm vol}^f(B_q(\eta))},\end{align}
 \begin{align}\label{lemma-3.5-2}
 \int_0^{l(y^*)} ds\int_0^{d(z^*,\gamma(s))}|{\rm hess }~h^+(\gamma_s(t))|dt
 \leq\epsilon(\frac{{\rm vol}^f(B_q(2))}{{\rm vol}^f(B_q(\eta))})^2,
 \end{align}
where  $\gamma_s(t)$  is the minimal geodesic curvse  connecting $ \gamma(s)$  and $z^*$.
\end{lem}

\begin{proof}
Choose a  cut-off function $\phi=\phi(\text{dist}(p,\cdot))$ with support in $B_p(1)$.  Let
\begin{align}
\nonumber e=\phi|\nabla h^+-\nabla b^+|,e_1=\phi|\text{hess }h^+|, \\
\nonumber e_2(y)=\int_{B_{q'}(\eta)}dz\int_0^{d(y,z)}e_1(\gamma_{yz})(s)ds.
\end{align}
Then by Lemma \ref{equ-rad}, we  have
\begin{align}\label{estimate-e}
 \int_{B_q(\eta)}\int_0^{d(q^-,y)}e(\gamma_{q^-y}(s))dsdy\leq c(A,\Lambda)\int_{B_p(1)}e(y)dy.
\end{align}
 On the other hand,  by  Lemma \ref{equ-seg}, one sees
\begin{align}
\nonumber \int_{B_p(1)}e_2(y)dy&=\int_{B_p(1)} dy\int_{B_{q'}(\eta)}dz\int_0^{d(y,z)}e_1(\gamma_{yz})(s)ds\\
&\leq  c_1(\Lambda, A)\text{vol }^f (B_p(1))\int_{B_p(1)}e_1(y)dy.\notag
\end{align}
Thus by  Lemma \ref{equ-rad},  we get
\begin{align}\label{estimate-e2}
&\int_{B_q(\eta)}\int_0^{d(q^-,y)}e_2(\gamma_{q^-y}(s))dsdy\notag\\
&\leq c_2(\Lambda,A)\int_{B_p(1)}e_2(y)dy\\
\nonumber &\leq  \text{vol}^f (B_p(1))c_3(\Lambda, A)\int_{B_p(1)}e_1(y)dy.
\end{align}
Observe that the left hand side of  (\ref{estimate-e2})  is equal to
\begin{align}
\nonumber \int_{B_{q}(\eta)}dy\int_{B_{q'}(\eta)}dz\int_0^{d(q^{-},y)}\int_0^{d(\gamma_{q^{-}y}(s),z)}e_1(\hat \gamma_s(t))dtds,
\end{align}
where $\hat\gamma_s(t)$ is the minimal geodesic from $z$ to $\gamma_{q^{-}y}(s)$ with arc-length parameter $t$.
Combining  (\ref{estimate-e}) and (\ref{estimate-e2}),  we find two points  $y^*,z^*$ such that   both  (\ref{lemma-3.5-1})  and (\ref{lemma-3.5-2}) are satisfied.

\end{proof}

Now we apply Lemma \ref{approxi-1} to   prove a local version of  Theorem \ref{splitting-theorem}.

\begin{prop}\label{proof-splitting} Let  $(M,g)$ be an $n$-dimensional   complete Riemannian manifold which satisfies
$${\rm Ric}_g^f\geq-\frac{n-1}{R^2}, ~|f|,~|\nabla f|\leq A.$$
Suppose that there exist three points $p,q^+,q^-$ such that
 \begin{align}\label{splitting-condition-1} d(p,q^{+})+d(p,q^{-})-d(q^{+},q^{-})<\epsilon\end{align}
 and
 \begin{align}\label{splitting-condition-2} d(p,q^{+})\geq R, d(p,q^{-}) > R.\end{align}
   Then there exists a map
\begin{align}
u:B_p(1/8)\longrightarrow B_{(0, x)}(1/8)
\end{align} as a $\Psi(1/R,\epsilon;A,n)$ Gromov-Hausdorff approximation,
where $B_{(0, x)}(1/8) \subset\mathbb{R}\times X$ is a $\frac{1}{8}$-radius ball centered at $(0,x) \in \mathbb{R}\times X$ and
 $X$ is given by  the level set  $(h^+)^{-1}(0)$  as a metric space   measured in the $B_p(1)$.
 \end{prop}

\begin{proof}
For simplicity,  we denote the terms on the  right-hand side of   (\ref{c0-h}), (\ref{h-gradient}) and (\ref{hessian-integral})  in Lemma \ref{harmonic-estimate}   by $\delta=\delta(\epsilon,\frac{1}{R})$.  Define a map $u$ on $B_p(1)$ by  $u(q)=(x_q,h^+(q))$,  where $x_q$ is the nearest point to $q$ in $X$. We are going to prove that $u$ is a $\Psi(1/R,\epsilon; A)$ Gormov-Hausdorff approximation.
Since $|\nabla h^{+}|\leq c=c(A)$ in $B_p(\frac{1}{2})$,
$$h^+(y)\leq 0,~ \forall~ y \in B_q(\eta), ~ \text{if}~ h^{+}(q)<-c\eta,$$
 where $\eta$ is an appropriate small number and it  will be  determined late. We call the area of $h^{+}(q)<-c\eta$ the upper region,  the area of $h^{+}(q)>c\eta$  the lower region and the rest the middle region, respectively.

Case 1.  Both points $q_1$ and $q_2$ in the upper region  ( we may assume that  $h^+(q_1)>h^+(q_2)$).  Let  $q$ be a point  in the upper region. Then  by  applying Lemma \ref{approxi-1}   to $q,x_q$,  we get a geodesic from a point $y$ near $q$ to $q^-$ whose direction is almost the same as $\nabla h^+$. Thus this geodesic must intersect $h^+=0.$  Applying  Triangle Lemma \ref{cheeger-lemma}, we  see  that the intersection is near $x_q$.  Hence  for  $q_1$ and $q_2$,  we can find $y_1$  and $y_2$ nearby $q_1$ and $q_2$ respectively,  such that two geodesics from $y_1$ and $y_2$ to $q^-$  intersect  $X$ with  points $x_1$ and  $x_2$, respectively. Denote the geodesic from $x_2$ to $y_2$ by $\gamma(s):\gamma(0)=x_2, \gamma(h^+(y_2))=y_2$. Applying  Triangle  Lemma \ref{cheeger-lemma} to  triples  $\{y_1,y_2,\gamma(h^+(y_1))\}, \{x_2,y_1,\gamma(h^+(y_1))\}$ and $\{x_1,x_2,y_1\}$, respectively,  we get
$$ |d(y_1,y_2)^2-|h^+(y_2)-h^+(y_1)|^2-d(y_1,\gamma(h^+(y_1)))^2|\leq c(n,A)\frac{\delta}{\eta^n},$$
$$|d(y_1,x_2)^2-d(y_1,\gamma(h^+(y_1)))^2-h^+(y_1)^2|\leq c(n,A)\frac{\delta}{\eta^n},$$
and
$$|d(y_1,x_2)^2-d(x_1,x_2)^2-h^+(y_1)^2| \leq c(n,A)\frac{\delta}{\eta^n}.$$
Combining the above three relations, we derive
\begin{align}\label{dist-appr}
|d(q_1,q_2)-d(u(q_1),u(q_2))|\leq c(n,A)\frac{\delta}{\eta^n}<<1
\end{align}
as $\delta=o(\eta^n)$.

Case 2.  $q_1$ is in the middle region and $q_2$ is in the upper region.  Note that  $x_q$ is near $q$  if $q$ is in the middle region.  Then we can find two points $y_1$ and $y_2$ near $q_1$ and $q_2$ respectively, such that  Triangle Lemma  \ref{cheeger-lemma} holds for  the triple $\{y_1,y_2,x_2\}$.  Hence for such two points $q_1$ and $q_2$,  we get (\ref{dist-appr}) immediately.

Case 3.  $q_1$ is in the lower region and $q_2$ is in the upper region.   As in  Case  1.  we can get one geodesic from $q^+$ to a point near $q_1$ and another geodesic from $q^-$ to a point near $q_2$, respectively.   Thus we can use same argument in Case 1 to obtain (\ref{dist-appr}).  Similarly, we can settle down another two cases,  both $q_1$ and $q_2$  in  the  lower region   and  both $q_1$ and $q_2$  in the middle region.

\end{proof}

\begin{proof}[Proof of  Theorem \ref{splitting-theorem}]

Suppose that  the line in $Y$ is $\gamma(t)$ and $\gamma(0)=y$.  Define a  Busemann  function  $b$ along $\gamma$ by
\begin{align}
\nonumber b(y)=\lim_{t\rightarrow +\infty}(d(y,\gamma(t))-t).
\end{align}
Since
\begin{align}
 \nonumber d_{GH}(B_{p_i}(j),B_y(j))\rightarrow 0, ~\text{as}~i\rightarrow \infty,
 \end{align}
  for any given  integer number $j>0$,  we may assume that
\begin{align}
\nonumber d_{GH}(B_{p_i}(j),B_y(j))<\frac{1}{j}, \epsilon_i<\frac{n-1}{j^2}\text{ for }i=i(j) \text{ large enough}.
\end{align}
Choose  a  Gromov-Hausdorff approximation from $B_y(j)$ to $B_{p_i}(j)$ so that  the images of endpoints $\gamma(j)$ and
$\gamma(-j)$ of the line in $B_y(j)$ together with  $p_i$  satisfy the conditions  (\ref{splitting-condition-1}) and  (\ref{splitting-condition-2})  in Proposition 3.6.
Then  we see that  there exist a metric space $X_j$ and a  Gromov-Hausdorff approximation  $u_j:B_{p_i}(1)\to B_{0\times x_j}(1)$  such that
\begin{align}
\nonumber d_{GH}( B_{p_i}(1), u_j(B_{p_i}(1))) <\Psi(\frac{1}{j}).
\end{align}
As a consequence,   there exists a map  $\hat u_j:B_{_y}(1)\to B_{0\times x_j}(1)$  such that
\begin{align}
\nonumber d_{GH}(B_y(1),\hat u_j(B_y(1)))<\Psi.
\end{align}
This implies that  all the projection of   $\mathbb{R}$ component from   space  $\mathbb{R}\times X_j$   are close to the Buseman function $b$ along the given line in $Y$ for $j>>1$,  so they are  almost the same.  Hence,  $\{ X_j\}$ is a Cauchy sequence in Gromov-Hausdorff topology with a limit $X$.  It follows that $B_y(1)=B_{0\times x}(1)$ where $x$ is the limit point of $\{x_j\}$ in $X$.  Since the number $1$ can be replaced  by any  positive number,  we finish  the proof of theorem.
\end{proof}

\vskip3mm

\section{Existence of  metric cone}

In this section, we prove    an anology of Theorem \ref{thm-cc1}  in  the  Bakry-\'Emery geometry.  Namely,  we prove 
the existence of metric cone of  a tangent space on   the  limit  space of a sequence in $\mathcal{M} (A,v,\Lambda)$. Recall

\begin{defi}
 For a metric space $(Y, d)$,  the  limit of $(Y, \epsilon_i^{-2
 }d;y)$  in  the Gromov-Hausdorff  topology as $\epsilon_i\to0$  is called a tangent cone of $Y$ at $y$ (if exists).
 We denote it  by  $T_yY$.
 \end{defi}

\begin{defi}\label{def-metric-cone}
Given a metric space $X$, the space $\mathbb{R}^+\times X$ with the metric defined by
\begin{align}
\nonumber &d((r_1,x_1),(r_2,x_2))=\sqrt{r_1^2+r_2^2-2r_1r_2\cos d(x_1,x_2)}, \text{ if }d(x_1,x_2)\leq \pi,\\
\nonumber &d((r_1,x_1),(r_2,x_2))=r_1+r_2, \text{ if }d(x_1,x_2)\geq \pi
\end{align}
is called a metric cone over $X$.  We usually   denote it  by $C(X)$ with the metric $\mathbb{R}^+\times_rX$.
\end{defi}

 The main theorem of this section can be stated as follows.

 \begin{theo}\label{existence-metric-cone}
Let $\{(M_i,g_i;p_i)\}$ be a sequence of manifolds in $\mathcal{M}(A,v,\Lambda)$. Then there exists a subsequence of $\{(M_i,g_i;p_i)\}$
  converges  to a metric space $(Y; y)$  in the pointed  Gromov-Hausdorff  topology.  Moreover,  for each
  $z\in (Y;y)$,  each tangent cone  $T_zY$ is a metric cone over another metric space whose diameter is less than $\pi$.
\end{theo}

 The proof of Theorem \ref{existence-metric-cone} is similar to one of Splitting Theorem \ref{splitting-theorem}.  We need another  triangle lemma to estimate  the distance.

\begin{lem}\label{cheeger-lemma-2}
Let $x,y$ be two points  in a minimal  geodesic    from $p$ and denote  the part of  the  geodesic  curve from  $x$ to $y$   by $\gamma(s)$. Let  $\gamma_s(t)$  be  a family of geodesic  curves  connecting $z$ and $\gamma(s)$  as in Lemma \ref{cheeger-lemma}.
Suppose that there is  a smooth function $h$  on $M$ which satisfies
\begin{align}&i)~ |h(z)-h(x)-\frac{r(z)^2-r(x)^2}{2}|<\delta<<1;\notag\\
&ii)~ \int_{[0,a]}|\nabla h(\gamma(s))-r(\gamma(s))\gamma'(s)|<\delta<<1;\notag\\
& iii)~\int_{[0,a]}\int_{[0,l(s)]}|{\rm hess }~h-g|dtds <\delta<<1.\notag
\end{align}
Here $r(\cdot)=\text{dist}(p,\cdot)$. Then
\begin{align}\label{cosine}
&d(z,y)^2r(x)-d(x,z)^2r(y)\\
\nonumber &+ r(z)^2(r(y)-r(x))-r(x)r(y)(r(y)-r(x)) <\epsilon(\delta).
\end{align}
\end{lem}

\begin{proof} The proof is similar to one of Lemma \ref{cheeger-lemma}.  First,  we have
\begin{align}
&|h(\gamma(s))-h(\gamma(0))-\frac{(s+r(x))^2}{2}+\frac{r^2(x)}{2}|\notag\\
 &=|\int_0^s\langle \nabla h(\gamma(s))-(s+r(x))\gamma'(s),\gamma'(s)\rangle|\leq \delta.\notag
\end{align}
Then
$$\label{vertical}  h(\gamma_s(l(s)))= h(\gamma(s))=h(x)+\frac{(s+r(x))^2}{2}-\frac{r^2(x)}{2}+o(1).$$
Since
\begin{align} l(s) h'(\gamma_s(0))&=h(\gamma_s(l(s)))-h(z)-\frac{l^2(s)}{2}\notag\\
&-\int_0^a \int_0^{l(s)} (\text{hess}h(\gamma_s'(t),\gamma_s'(t)) -g(\gamma_s'(t),\gamma_s'(t)))  dtds,\notag
\end{align}
from the condition iii) and i),  we get
 \begin{align}l(s)h'(\gamma_s(0))&=\frac{(s+r(x))^2}{2}-\frac{r^2(x)}{2}+h(x)-h(z)-\frac{l^2(s)}{2}+o(1)\notag\\
&=\frac{(s+r(x))^2}{2}-\frac{r^2(z)}{2}-\frac{l^2(s)}{2}+o(1).\notag
\end{align}
Consequently,  we obtain
\begin{align}
l(s)h'(\gamma_s(l(s)))&=\frac{(r(x)+s)^2-r^2(z)}{2}+\frac{l^2(s)}{2}\notag\\
&+ l(s)\int_0^{l(s)} (\text{hess }h(\gamma_s'(t),\gamma_s'(t)) -g(\gamma_s'(t),\gamma_s'(t)))  dt+
 o(1).\notag
\end{align}
Hence we derive
\begin{align}\label{short-distance-2}
&\int_0^a(\frac{2l(s)h(\gamma_s'(l(s)))}{(s+r(x))^2}-\frac{l^2(s)}{(s+r(x))^2}) ds\notag\\
& =a+\frac{r^2(z)}{r(x)+a}-\frac{r^2(z)}{r(x)}+o(1).
\end{align}

 Secondly,  by
 the first variation formula,
$$
l'(s)=\langle\gamma_s'(l(s)),\gamma'(s)\rangle,
$$
 we get from the condition ii),
\begin{align}
\int_0^a(\frac{l^2(s)}{s+r(x)})' ds&=\int_0^a(\frac{2l(s)l'(s)}{s+r(x)}-\frac{l^2(s)}{(s+r(x))^2}) ds\notag\\
&=\int_0^a(\frac{2l(s) (s+r(x))\langle \gamma_s'(l(s)),\gamma'(s)\rangle}{(s+r(x))^2}-\frac{l^2(s)}{(s+r(x))^2})ds \notag\\
&=\int_0^a(\frac{2l(s)\langle\gamma_s'(l(s)),\nabla h(\gamma(s))\rangle}{(s+r(x))^2}-\frac{l^2(s)}{(s+r(x))^2})ds+o(1)\notag\\
  &=\int_0^a (\frac{2l(s)h'(\gamma_s(l(s)))}{(s+r(x))^2} - \frac{l^2(s)}{(s+r(x))^2}) ds +o(1)\notag.
   \end{align}
Therefore,  by  combining (\ref{short-distance-2}), we get (\ref{cosine})  immediately.

\end{proof}

It is easy to see  the left-hand side of (\ref{cosine}) is zero  in a metric cone $C(X)$ if $x,y$ lie in a radial direction.
We need a few of  lemmas more  to prove Theorem \ref{existence-metric-cone}.

\begin{lem}\label{distance-appro-3}
Given $\eta>0$, there exists $\omega=\omega(a,b,\eta, A,\Lambda)$ such that the following is true:
if
 \begin{align}\label{ricci-condition-2}
{\rm Ric}_{g}^f\geq-(n-1)\Lambda^2 g~\text{and}~|\nabla f|\leq A\end{align}
and
\begin{align}\label{volume-condition-2}
\frac{{\rm vol}^f(\partial B_p(b))}{{\rm vol}^f(\partial B_p(a))}\geq(1-\omega)\frac{L_{\Lambda, A}(b)}{L_{\Lambda, A}(a)},
\end{align}
then  for any point $q$ on $\partial B_p(a)$, there exists $q'$ on $\partial B_p(b)$ such that
$$d(q,q')\leq b-a+\eta.$$
 \end{lem}

\begin{proof}
Suppose that  the conclusion fails to hold for some $\eta$ and $q_1\in \partial B_p(a)$.  Then for any point in $B_{q_1}(\frac{\eta}{3})$, there is no point $q$ on $\partial B_p(b)$ such that $d(q_1,q)\leq b-a+\frac{\eta}{3}$.
Thus for any $r<\frac{\eta}{3}$, any minimal geodesic from $p$ to $\partial B_p(b)$ does not  intersect with $B_{q_1}(\frac{\eta}{3})\cap \partial B_p(a+r)$. Since
\begin{align}
{\rm vol}^f(B_{q_1}(\frac{\eta}{3}))\geq \frac{L_{\Lambda, A}(\frac{\eta}{3})}{L_{\Lambda, A}(2b)}{\rm vol}^f(A_p(a,b)),
\notag\end{align}
by the coarea formula,
 there exists  some $\frac{\eta}{4}<r<\frac{\eta}{3}$ such that
\begin{align}
{\rm vol}^f( B_{q_1}(\frac{\eta}{3})\cap \partial B_p(a+r))\geq \frac{1}{\eta}\frac{L_{\Lambda, A}(\frac{\eta}{3})}{L_{\Lambda, A}(2b)}{\rm vol}^f(A_p(a,b)).\notag
\end{align}
  Using  the monotonicity  formula (\ref{mono-formula}), we get
\begin{align}
{\rm vol}^f(\partial B_p(b))
&\leq {\rm vol}^f(\partial B_p(a+r)\setminus B_{q_1}(\frac{\eta}{3}))\frac{L_{\Lambda, A}(b)}{L_{\Lambda, A}(a+r)}\notag\\
&\leq ({\rm vol}^f(\partial B_p(a+r))-\frac{1}{\eta}\frac{L_{\Lambda, A}(\frac{\eta}{3})}{L_{\Lambda, A}(2b)}{\rm vol}^f(A_p(a,b)))\frac{L_{\Lambda, A}(b)}{L_{\Lambda, A}(a+r)}.\notag
\end{align}
It follows
\begin{align}
{\rm vol}^f(\partial B_p(b))&\le (1+\delta' (\eta,b,a))^{-1} {\rm vol}^f(\partial B_p(a+r) )
\frac{L_{\Lambda, A}(b)}{L_{\Lambda, A}(a+r)}\notag\\
&\le  (1+\delta' (\eta,b,a))^{-1} {\rm vol}^f(\partial B_p(a) )
\frac{L_{\Lambda, A}(b)}{L_{\Lambda, A}(a)}.\notag
\end{align}
But this is  a contradiction to   (\ref{volume-condition-2}) as $\omega<\frac{1}{2}\delta' (\eta,b,a)$.  Therefore,  the lemma is proved.
\end{proof}

By applying  Theorem 3.6 in $[CC1]$ with the help  of Lemma \ref{cheeger-lemma-2} and  Lemma \ref{distance-appro-3},
 we have the following proposition.

\begin{prop}\label{segment-inequ-2}
Given $\eta>0$, there exist $\omega=\omega(a,b,\eta)$  and $\delta=\delta(\eta)$ such that if
(\ref{ricci-condition-2}) and (\ref{volume-condition-2}) are satisfied,
then there is a length space $X$ such that
$$d_{GH}(A_p(a,b),(a,b)\times_rX) <\eta,$$
where $(a,b)\times_rX$ is an  annulus in $C(X)$ and the metric of $A_p(a,b)$ is measured in a slightly bigger annulus in $M$.
\end{prop}

\begin{proof} It  suffices  to verify  the  condition for distance function in  Theorem 3.6 in $[CC1]$ .
Let   $x,y,z,w$ be four points  in the annulus $A_p(a,b)$  such that both  pairs $\{x,y\}$ and $\{z,w\}$ are in the radial direction from $p$.   Then by  applying the segment inequality   of  Lemma \ref{approxi-1} to the function $h$  in Lemma \ref{harmonic-estimate-annual-1} and Lemma \ref{harmonic-estimate-annual-2},  we can find  another four points $x_1,y_1,z_1,w_1$ near the four points respectively  such that  Triangular Lemma \ref{cheeger-lemma-2} holds  for two triples   $\{x_1,y_1,z_1\}$ and   $\{y_1,z_1,w_1\}$.
Now we choose four points $x_2,y_2,z_2,w_2$ in the plane $\mathbb R^2$  such that both triples  $ \{O, x_2, y_2\}$  and $\{O,z_2, w_2\}$ are  co-linear.  Moreover,  we  can require that
$$r(x_2)=r(x_1),r(y_2)=r(y_1),r(z_2)=r(z_1),r(w_2)=r(w_1)$$
and
$$d(x_1,z_1)=d(x_2,z_2).$$
Thus  by  using Triangle Lemma \ref{cheeger-lemma-2} to $\{x_1,y_1,z_1\}$ ,  it is easy to see that
\begin{align}\label{dis-1}
 |d(y_2,z_2)-d(y_1,z_1)|< \Psi.
\end{align}

Applying  Triangle Lemma \ref{cheeger-lemma-2} to $\{y_1,z_1,w_1\}$ , we  have
\begin{align}\label{dis-2}
&|d(y_1,z_1)^2r(w_1)+r(w_1)r(z_1)(r(w_1)-r(z_1))\notag\\
&-d(y_1,w_1)^2r(z_1)-r(y_1)^2d(z_1,w_1)|< \Psi.
\end{align}
Note that  the left hand side of  (\ref{dis-2})  is zero when  the triple $\{y_1,z_1,w_1\}$  is  replaced by  $\{y_2,z_2,w_2\}$  in the plane.
Since
$$|d(z_1,w_1)-(r(w_1)-r(z_1))|< \Psi,$$
we get from  (\ref{dis-1}) and  (\ref{dis-2}) that,
\begin{align}
\nonumber |d(y_1,w_1)-d(y_2,w_2)|< \Psi.
\end{align}
On the other hand,  $d(y_2,w_2)$ can be written  as  the following function:
\begin{align}
d(y_2,w_2)=Q(r(x_2),r(y_2),r(z_2),r(w_2), d(x_2,z_2)). \notag
\end{align}
Therefore
\begin{align}
|d(y_1,w_1)-Q(r(x_1),r(y_1),r(z_1),r(w_1), d(x_1,z_1))|< \Psi.\notag
\end{align}
It follows that
\begin{align}\label{distance-function-condition}
|d(y,w)-Q(r(x),r(y),r(z),r(w),d(x,z))| <\Psi.
\end{align}
(\ref{distance-function-condition})  is just   the condition  for  distance function  in  Theorem 3.6 in [CC1].

By (\ref{distance-function-condition}) and Lemma \ref{distance-appro-3} we see that two conditions in Theorem 3.6 in [CC1] are satisfied.  Hence as a consequence of this theorem,  we obtain Proposition \ref{segment-inequ-2}.  In fact,
 $X$ is a level set  of  $r^{-1}(a)$ with a $\chi$-intrinsic metric defined by
\begin{align}
l^\chi(x,y)=\frac{1}{a}\inf\Sigma_{i=1}^n d(x_{i-1},x_i),
\end{align}
 where the infimum  is taken among all the sequences $\{x_i\} \in X$  which satisfy  $x_0=x, x_n=y$ and  $d(x_{i-1},x_i)\leq \chi$.

\end{proof}

It remains   to verify  the condition (\ref{volume-condition-2}) in Lemma \ref{distance-appro-3}.

\begin{lem}\label{volume-sphere}
 Given $0<a<b=a\Omega, \Omega>0$, there exists an integer $N=N(n,\Omega,\Lambda, v,A)$ such that for any sequence of $r_i$ $(1\leq i\leq N)$ with  $\Omega r_{i+1}\leq r_i\leq \frac{1}{b}$,   the volume condition (\ref{volume-condition-2}) for any manifold
 $(M,g)\in\mathcal{M}(\Lambda,v,A)$ in Lemma \ref{distance-appro-3} holds for some annulus $A_p(ar_k,br_k)\subset M $ $(1\leq k\leq N)$
with   rescaling metric $\hat g=\frac{g}{r_k}$.
  \end{lem}

 \begin{proof}
We only need to give an upper bound of $N$   in case that  the following inequality
 \begin{align}\label {vol-con}
\frac{{\rm vol}_{\hat g}^f(\partial B_p(br_k))}{L_{r_k\Lambda,r_kA}(br_k)}\geq e^{-\omega}\frac{{\rm vol}_{\hat g}^f(\partial B_p(ar_k))}{L_{r_k\Lambda,r_kA}(ar_k)}
 \end{align}
 doesn't hold for any $1\leq k\leq N$. Then by the   monotonicity formula (\ref{mono-formula}),  we know that
 \begin{align}
 \nonumber \frac{{\rm vol}_{\hat g}^f(\partial B_p(br_N))}{L_{r_k\Lambda,r_kA}(br_N)}\leq e^{-N\omega}\frac{{\rm vol}_{\hat g}^f(\partial B_p(br_1))}{L_{r_k\Lambda,r_kA}(br_1)}.
 \end{align}
   Thus by the non-collapsing condition the left-hand side has a lower bound $c_1(n, \Lambda,v,A)$, and by    Volume Comparison  Theorem \ref{volume-comparison} the right-hand side is not greater than $e^{-N\omega}c_2(n, \Lambda, v,A)$. Thus this helps us to  get an upper bound of $N$.  Hence,  if $N$ is larger than this bound, there must be some $k$ such that (\ref{vol-con}) holds. The lemma is proved.
\end{proof}

\begin{proof}[Proof of Theorem\ref{existence-metric-cone}]

Without loss of generality, we may assume that $z=y$  since each point in $(Y,d;y)$ is a limit of sequence of  volume non-collapsing  points in $M_i$.  Also we note that  the  tangent cone   $T_yY$ always exists in our case by Gromov's theorem [Gr].   By the contradiction argument,  we  suppose that $T_yY$ is not a metric cone.  Then  it is easy to see that there exist  numbers  $0<a<b,\eta_0>0$   and a sequence $\{r_i\}$,  which  tends  to $0$, such that for any  length space $X$ annulus  $A_y(ar_i,br_i)\subset (Y, \frac{d}{r_i};y)$  satisfy,
\begin{align}\label{equ-1}
d_{GH}(A_y(ar_i,br_i),(ar_i,br_i)\times_rX)>3r_i\eta_0.
\end{align}
By  taking  a subsequence we may assume that $\Omega r_{i+1}\leq r_i$ $(\Omega=\frac{b}{a})$ and $r_i$ is smaller than $\delta$ in Lemma \ref{distance-appro-3}.  On the other hand,  since $Y$ is the limit of $M_i$, we can find an increasing sequence $m_i$ such that for every $j\geq m_i$
 \begin{align}\label{anulus-close-3}
 d_{GH}(A_y(ar_i,br_i),A_{p_j}(ar_i,br_i))<r_i\eta_0.
 \end{align}
Let $\omega$ be a small  number as chosen in Proposition \ref{segment-inequ-2} and $N$ an   integer such that  Lemma \ref{volume-sphere} is true for the $\omega>0$. Thus by (\ref{anulus-close-3}),  we see that there exist a  subsequence $\{r_{i_k}\} \to 0$ and a
sequence $\{j_k\}\to\infty$ such that
\begin{align}\label{anulus-close-4}
 d_{GH}(A_y(ar_{i_k},br_{i_k}),A_{p_{j_k}}(ar_{i_{k}},br_{i_{k}}))<r_{i_k}\eta_0,
 \end{align}
where   annulus  $A_{p_{j_{k}}}(ar_{i_k},br_{i_k})$ are chosen  as in   Lemma \ref{volume-sphere}.  Now we can apply   Proposition \ref{segment-inequ-2} to show that  for each large $k$ there exists    a length space $X$ such that
\begin{align}
\nonumber d_{GH}(A_{p_{j_{k}}}(ar_{i_k},br_{i_k}),(ar_{i_k},br_{i_k})\times_rX)<r_{i_k}\eta_0.
\end{align}
But  this is impossible by (\ref{equ-1}).  Therefore,   $T_yY$ must be a metric cone.

 The diameter estimate  follows  from Splitting Theorem \ref{splitting-theorem}.   In fact,   if $diam(X)>\pi$, there will be two points  $p,q$ in $X$ such that  $d(p,q)=\pi$.  By  Theorem \ref{splitting-theorem},  it follows  that $C(X)=\mathbb{R}\times Y_1$, where $Y_1$ is also a metric cone,  i.e. $Y_1=C(X_1)$.  It is clear that  $diam(X_1)>\pi$ since  $diam(X)>\pi$. Thus we  can  continue to apply  Theorem \ref{splitting-theorem} to split off $X_1$.  By the induction,   $ C(X)$  should be an Euclidean space,  and consequently  $X$ is a standard  sphere . But this is impossible by the assumption that $diam(X)>\pi$.

  \end{proof}

Following the argument in the proofs  of  Theorem \ref{existence-metric-cone} and Proposition \ref{segment-inequ-2},   we   actually prove the  following strong  approximation of Gromov-Hausdorff    to  the flat space.

\begin{cor}\label{hausddorf-closed}
For all $\epsilon>0$,  there exists $\delta=\delta(n,\epsilon),\eta=\eta(n,\epsilon)$ such that   if
\begin{align}\label{small-curvature-2}
{\rm Ric}^f_{g}\geq-(n-1)\delta^2 g, ~|\nabla f|\leq\eta
\end{align}
and
\begin{align}\label{volume-condition-3} e^{-f(0)}{\rm vol}^f(B_p(1))\geq(1-\delta){\rm vol}(B_0(1))
\end{align}
are satisfied,  then
\begin{align}\label{conclusion-hausddorf-closed}
 d_{GH}(B_p(1),B_0(1))<\epsilon.
\end{align}
\end{cor}

\begin{proof}

Suppose that the conclusion (\ref{conclusion-hausddorf-closed})  is not true.  Then  there exist   sequences of $\{\delta_i\}$ and
$\{\eta_i\}$  which  tend $0$ both,   and a sequence of manifolds $\{(M,g_i)\}$  with conditions (\ref{small-curvature-2}) and (\ref{volume-condition-3}) such that
 \begin{align}\label{conclusion-hausddorf-unclosed}
  d_{GH}(B_{p_i}(1),B_0(1))\ge \epsilon_0>0,
 \end{align}
 where  $B_{p_i}(1)\subset M_i$.   Then following  the argument in the proofs  of  Theorem \ref{existence-metric-cone} and Proposition \ref{segment-inequ-2},   it is no hard to show  that $B_{p_i}(1)$ converge to a  limit  $B_x(1)$ which  is a  metric ball with radius $1$ in  a metric-cone $(C(X),d)$ with vertex $x$.   Since the blowing-up  space of $B_x(1)$ at $x$ is $C(X)$ itself,    we  see that
 there are  subsequences $\{j\}$ and  $\{i_j\}$, both of  which   tend to infinity,  such that
\begin{align}
\nonumber (B_{p_{i_j}}(j),j^2g_{i_j},q_{i_j})\rightarrow (C(X),d,x).
\end{align}

For any  $y\in X$, we choose a sequence of points $q_{i_j} \in B_{p_{i_j}}(j) \subset (M_{i_j},j^2g_{i_j})$  which tends to $y$.  Then for any given $R>0$, we have
 \begin{align}
 \nonumber B_{q_{i_j}}(R)(\subseteq (M_{i_j}, j^2g_{i_j})) \rightarrow B_y(R).
 \end{align}
 Since  the volume condition  (\ref{volume-condition-3}) implies
 \begin{align}
 e^{-f(0)}{\rm vol}^f(B_{q_{i_j}}(R))\rightarrow {\rm vol}(B_0(R)),
 \end{align}
by the above argument,  $B_y(R)$ is  in fact a metric ball  with radius $R$  in a metric cone $C(Y)$ with vertex $y$.   Note that $R$ is arbitrary.  We  prove that  $C(X)$ is  also a cone with vertex at $y$.  This shows  that there exists  a line connecting $x$ and $y$ in $C(X)$.   By Splitting Theorem \ref{splitting-theorem},  $C(X)$ can split off a line along the direction $xy$.   Since   $y \in X$ can be taken in  any direction,  $C(X)$ must be an euclidean space. But this is impossible according to  (\ref{conclusion-hausddorf-unclosed}).   The Corollary  is proved.

\end{proof}

\begin{rem} Corollary \ref{hausddorf-closed} is a generalization of    Theorem 9.69 in [Ch2]  in   the Bakry-\'{E}mery  geometry.  It  will be used in Section 5 and Section 6 for the  blowing-up analysis.  We also note that
 $e^{-f(0)}{\rm vol}^f(B_p(1))$ is close to ${\rm vol}(B_p(1))$  since $|\nabla f|$ is small enough.  Thus the volume condition
 (\ref{volume-condition-3})
 can be replaced by
 $${\rm vol}(B_p(1))\ge (1-\delta){\rm vol}(B_0(1)).$$
   \end{rem}

For the rest of this section,  we  prove   the Colding's  volume convergence  theorem   in    the Bakry-\'{E}mery  geometry  by using  the Hessian estimates in Section 2 [Co3].

\begin{theo}\label{volume-convergence}
Let $(M_i^n,g_i)$ be a sequence of  Riemannian manifolds  which satisfy (\ref{ricci-condition-2}).
Suppose that  $M_i$ converge to an $n$-dimensional  compact manifold $M$ in  the Gomov-Hausdorff  topology.  Then
$$\lim_{i\to \infty}{\rm vol}(M_i,g_i)= {\rm vol}(M).$$
\end{theo}

We first prove a local version  of Theorem \ref{volume-convergence} as follows.

\begin{lem}\label{volume-estimate-4}
Given $\epsilon>0$, there exist $R=R(\epsilon,\Lambda,A,n)>1$  and $\delta=\delta(\epsilon,\Lambda,A,n)$ such that if
\begin{align}
{\rm Ric}^f_{M,g}\geq-(n-1)\frac{\Lambda^2}{R^2} g,|\nabla f|\leq\frac{A}{R},
\end{align}
 and
 \begin{align}d_{GH}(B_p(R),B_0(R))<\delta,\end{align}
  then we have
\begin{align}
{\rm vol}(B_p(1))> {\rm vol}(B_0(1))-\epsilon.
\end{align}
\end{lem}

\begin{proof}
We  need to construct a Gromov-Hausdorff approximation  map by  using $f$-harmonic  functions  constructed  in Section 2.  Choose
$ n $ points $q_i$ in $B_p(R)$  which is close to $Re_i$ in $B_0(R)$, respectively.
Let  $l_i(q)=d(q,q_i)-d(q_i,p)$ and $h_i$ a solution of
 \begin{align}\Delta^fh_i=0, ~\text{in}~B_1(p), \notag\end{align}
with  $ h_i=l_i$  on  $\partial B_1(p)$.
Then by Lemma \ref {harmonic-estimate},  we have
\begin{align}
 \nonumber \frac{1}{\text{vol }(B_p(1))}\int_{B_p(1)}|\text{hess }h_i|^2 <\Psi(1/R,\delta;A).
\end{align}
By using  an  argument  in [Co3] (cf.  Lemma 2.9),  it follows
\begin{align}\label{orthorgonal-4}
\frac{1}{\text{vol }(B_p(1))}\int_{B_p(1)}|\langle\nabla h_i,\nabla h_j\rangle-\delta_{ij}| <\Psi(1/R,\delta;A).
 \end{align}
Define a map by  $h=(h_1,h_2,...,h_n)$. It is easy to see that  the map $h$ is a $\Psi(\frac{1}{R}, \delta; \Lambda)$ Gromov-Hausdorff approximation to $B_p(1)$  by using the estimate (\ref{c0-h})  in Lemma \ref{harmonic-estimate}.
  Since $h$
 maps  $\partial B_p(1)$  nearby  $\partial B_0(1)$ with distance  less than $\Psi$,  by a small modification  to $h$ we may assume that
 \begin{align}
  \nonumber h:(B_p(1),\partial B_p(1))\longrightarrow(B_0(1-\Psi),\partial B_0(1-\Psi)).
  \end{align}

Next we use  a degree  argument  in [Ch2]  to show that the image of $h$ contains $B_0(1-\Psi)$. 
  By using Vitali covering lemma,   there exists a point $x$ in $B_p(\frac{1}{8})$ such that for any $r$  less than $\frac{1}{8}$ it holds
  \begin{align}\label{almost-hassian-zero-2}
  \frac{1}{\text{vol }(B_x(r))}\int_{B_x(r)}|\text{hess }h_i|<\Psi
  \end{align}
  and
  \begin{align}\label{orthogonal-2}
   \frac{1}{\text{vol }(B_x(r))}\int_{B_x(r)}|\langle\nabla h_i,\nabla h_j\rangle-\delta_{ij}|<\Psi.
  \end{align}
    Let  $\eta=\Psi^{\frac{1}{2n+1}}$.  For any $y$ with  $d(x,y)=r<\frac{1}{8}$,
  applying Lemma \ref{segment-inequ} to $A_1=B_x(\eta r), A_2=B_y(\eta r), e=|\text{hess }h_i|$ , we get from (\ref{almost-hassian-zero-2}),
  \begin{align}
&\int_{B_x(\eta r)\times B_y(\eta r)}\int_{\gamma_{zw}}|\text{hess }h_i(\gamma',\gamma')|\notag\\
 &< r(\text{vol }(B_x(\eta r))+\text{vol }(B_y(\eta r)))\text{vol }(B_x(r)) \Psi\notag.&
  \end{align}
It follows that
 \begin{align}
&\int_{B_x(\eta r)} [  Q(r,\eta) \int_{ B_y(\eta r)}  \int_{\gamma_{zw}} \Sigma_{i=1}^n |\text{hess }h_i(\gamma',\gamma')|
+  | \langle\nabla h_i,\nabla h_j\rangle-\delta_{ij}|]
\notag\\
 &<\text{vol }   (B_x(\eta r)) \Psi\notag,&
  \end{align}
 where $Q(r,\eta)=\frac{{\rm vol}B_x(\eta r)}{ r(\text{vol }(B_x(\eta r))+\text{vol }(B_y(\eta r))){\rm vol}B_x(r)}$.
  Consider
  $$  Q(r,\eta) \int_{ B_y(\eta r)} \int_{\gamma_{zw}}\Sigma_{i=1}^n\text{hess } |h_i(\gamma',\gamma')| + | \langle\nabla h_i,\nabla h_j\rangle-\delta_{ij}|$$
  as a function of $z\in B_x(\eta r)$.   Then  one sees that  there exists a  point $x^*\in B_x(\eta r)$ such that
 \begin{align}\label{orthogonal-3}
 |\langle\nabla h_i,\nabla h_j\rangle(x^*)-\delta_{ij}|< \Psi
 \end{align}
and
 \begin{align}\label{hession-small-geodesic}
\Sigma_{i=1}^n\int_{B_y(\eta r)} \int_{\gamma_{x^*w}}|\text{hess }h_i(\gamma',\gamma')| < r\text{vol }(B_x(r))\eta^{-n}\Psi.
 \end{align}
Here at the  last inequality, we used  the volume comparison (\ref{volume-estimate-2}).
 Moreover  by  (\ref{hession-small-geodesic}),   we  can find  a point  $y^*\in B_y(\eta r)$ such that
\begin{align}\label{almost-hessian-zero-3}
  \Sigma_{i=1}^n\int_{\gamma_{x^*y^*}}|\text{hess }h_i(\gamma',\gamma')| <\eta r.
\end{align}

  By a direct calculation with help of   (\ref{orthogonal-3}) and (\ref{almost-hessian-zero-3}),    we get
 \begin{align}\label{dist-appr-1}
 (h(x^*)-h(y^*))^2=(1+\Psi^{\frac{1}{2n+1}})r^2.
 \end{align}
 This shows that $ h(x)\neq h(y)$ for any $y$ with $d(y,x)\leq \frac{1}{8}$.  On the other hand, for any $y$ with $d(y,x)\geq \frac{1}{8}$,  it is clear that $h(x)\neq h(y)$ since  $h$ is a $\Psi$ Gromov-Hausdorff approximation.
 Thus  we prove that   the pre-image of $h(x)$ is unique.  Therefore the degree of $h$ is  $1$,  and consequently,
  $B_0(1-\Psi)\subset h(B_p(1))$.     The lemma  is proved  because the volume of $B_p(1)$ is almost same to one of $h(B_p(1))$
  by the fact  (\ref{orthorgonal-4}).
\end{proof}

\begin{proof}[Proof of Theorem \ref{volume-convergence}]
Choose finite $r_i$-balls $B(q_i, r_i)$  to  cover $M$  with $r_i$ small enough to make all  balls close to Euclidean balls so that
$$\Sigma _i{\rm vol}(B(q_i, r_i))<(1+\epsilon){\rm vol}(M)$$
 for any given $\epsilon>0$.  Then for $j$ sufficiently large, $M_j$ can be covered by $B(q_{ji}, r_{ji})$ with $r_{ji}\leq(1+\epsilon)r_i$. Thus     by  the volume comparison   (\ref{volume-estimate-2}),  we have
  \begin{align}
  {\rm vol}(M_j)&\leq \Sigma_i{\rm vol}(B(q_{ji}, r_{ji}))\notag\\
  &< (1+\Psi(\delta:\Lambda,  A))\Sigma_i{\rm vol}(B(q_{i}, r_{i})).
   \end{align}
 Here $\delta={\rm max}\{r_i\}$.  Hence  we get
 $$\lim_{j\rightarrow\infty}{\rm vol}(M_j)\leq {\rm vol}(M).$$

 On the other hand,  for any $\epsilon>0$,  we choose small enough $N$ disjoint balls $B(q_i, r_i)$ in $M$ with $B(q_i, r_i)$ close to Euclidean balls  so that
 \begin{align}\label{small-ball-volume-2}
(1+\epsilon)\Sigma _i \omega_n r_i^n \ge  \Sigma _i{\rm vol}(B(q_i, r_i)) >(1-\epsilon){\rm vol}(M).
 \end{align}
 Then for a fixed  large number  $R$, we see that for $j$ large enough there are  corresponding   disjoint balls $B(q_{ij}, r_i)$ in $M_j$ such that $B(q_{ij}, Rr_i)$ is $\delta(N)$-close to $B(q_i, Rr_i)$ in Gromov-Hausdorff topology, where $\delta(N)$ is the number determined in Lemma \ref{volume-estimate-4} when $\epsilon$ is replaced by $\frac{\epsilon}{N}$.
 Apply the above lemma   to each  ball $B(q_{ij}, Rr_i)$ with rescaling  metric $\frac{g_j}{r_i}$,  we get from
 (\ref{small-ball-volume-2}),
 $$(1+\epsilon){\rm vol}(M_j)>(1-\epsilon){\rm vol}(M)-(1+\epsilon)\epsilon.$$
Taking $\epsilon$ to $0$ and $N$ to $\infty$,  it follows
 $$\lim_{j\rightarrow\infty}{\rm vol}(M_j)\ge {\rm vol}(M).$$
The theorem is proved.

 \end{proof}

\vskip3mm
\section {Structure of singular set I:  Case of Riemannian metrics}

According to  Theorem \ref{existence-metric-cone},  we  may  introduce a notion of   $\mathcal S_k$-typed  singular point $y$ in   the
 limit space $(Y,d_\infty;p_\infty)$ of  a sequence of Riemannian manifolds $\{(M_i,g_i;p_i)\}$  in $\mathcal{M}(A,v,\Lambda)$ as Definition  \ref{singular-type},   if  there exists  a tangent cone at $y$ which can be  split   out an euclidean space   $\mathbb R^k$ isometrically  with dimension at most $k$.  By applying  Metric Cone Theorem \ref{existence-metric-cone} to appropriate tangent cone spaces  $T_yY$,   we can follow the  argument in   [CC2] to show that   dimension of  $\mathcal{S}_k$    is less than $k$.  Moreover,      $\mathcal{S}=\mathcal{S}(Y)=\mathcal{S}_{n-2}$, where  $\mathcal{S}(Y)= \cup_{i=0}^{n-1}\mathcal{S}_i$.   The latter  is equivalent to  that  any tangent cone can't be the upper half space,   which  can be proved  by using a topological argument as in  the  case of  Ricci  curvature bounded below (cf. Theorem 6.2 in  [CC2]).
 Thus we have

 \begin{theo}\label{dimension-k}
 Let $\{(M_i, g_i; p_i)\}$ be a sequence of Riemannian manifolds   in  $\mathcal{M}(A,v,\Lambda)$ and  let $(Y,d_\infty;p_\infty)$  be its limit in  the Gromov-Hausdorff topology.
  Then $\text{dim } \mathcal{S}_k\leq k$ and $\mathcal{S}(Y)=\mathcal{S}_{ n-2}$.
 \end{theo}

 \begin{rem}\label{volume-covergerce-4}  By Theorem $\ref{dimension-k}$,    one sees that  $\mathcal{H}^n( \mathcal{S})=0$.
   Thus by  Theorem \ref{volume-convergence},  we have
  \begin{align}\label{volume-covergerce-5}
 \lim_{i\rightarrow\infty}{\rm vol}(M_i)=\mathcal{H}^n(Y).
 \end{align}
  Moreover, if  $B_i(r)\subset M_i$ converge to
$B_\infty(r)\subset Y$,
 \begin{align}\label{volume-cone}
\lim_{i\rightarrow\infty}{\rm vol}(B_i(r))=\mathcal{H}^n(B_\infty(r)),\end{align}
where $B_i(r)$  and $B_\infty(r)$  are   radius $r$-balls in $M_i$ and $Y$, respectively.
% where $B_{x^*}(u )$ is radius $u$-ball in $C(X)$ and ${\rm vol}(X)$ is an  $(n-1)$-dimensional Hausdorff measure of $X$.
 \end{rem}

We define  $\epsilon$-regular points in $Y$.

\begin{defi} $y\in (Y;p_\infty)$  is called an $\epsilon$-regular point if there exist an  $\epsilon$ and a sequence $\{r_i\}$ such  that
$${\rm dist}_{GH}((B_{y}(1),\frac{1}{r_i}d_\infty), B_0(1))<\epsilon,~\text{as}~i\to \infty.$$
 Here $B_0(1)$ is the unit ball in  $\mathbb{R}^{n}$. We denote  the set of those points by $\mathcal{R}_{\epsilon}$.
  \end{defi}

 In this section,   our main purpose is to prove   an anology of Theorem \ref{thm-cct}  in  the  Bakry-\'Emery geometry.

\begin{theo}\label{dimension-n-4}  Let $\{(M_i, g_i; p_i)\}$ be a sequence in  $\mathcal{M}(A,v,\Lambda)$ and $(Y;p_\infty)$
  its limit  as in  Theorem \ref{dimension-k}.  Suppose that
   \begin{align}\label{integral-curvature}
    \frac{1}{{\rm vol}(B_{p_i}(2))}\int_{B_{p_i}(2)}|\rm Rm|^p<C.\end{align}
    Then for any $\epsilon>0$,  the  following is true:  i)
      \begin{align} \label{p-les-2} \mathcal{H}^{n-2p}(B_{p_\infty}(1)\setminus \mathcal{R}_{2\epsilon})<\infty,~ \text{if}~ 1\leq p<2;\end{align}
ii)
\begin{align}\label{p=2}
     {\rm dim} (B_{p_\infty}(1)\setminus \mathcal{R}_{2\epsilon})\leq n-4,~\text{ if}~ p=2.\end{align}

\end{theo}

The  theorem  is a consequence of following result of   $\epsilon$-regularity.

\begin{prop}\label{epsilon-regularity-1}
For any $v,\epsilon>0$, there exist three small numbers $\delta=\delta(v,\epsilon,n)$, $\eta=\eta(v,\epsilon,n)$,  $\tau=\tau(v,\epsilon,n)$ and a big number  $l=l(v,\epsilon,n)$
  such that if $(M^n,g)$ satisfies
\begin{align}\label{condition-1-regularity}
{\rm Ric}^f_{M,g} >-(n-1)\tau^2, |\nabla f|<\tau, {\rm vol }(B_p(1))\geq v,\end{align}
\begin{align}\label{curvature-int}
\frac{1}{{\rm vol}(B_{p}(3))}\int_{B_p(3)}|{\rm Rm}|<\delta,
\end{align}
and for some metric space $X$,
\begin{align}\label{g-h-close}
{\rm d}_{GH}(B_{p}(l),B_{(0,x)}(l)) <\eta
\end{align}
 holds for $k=2$ or $3$, where $(0,x)$ is the vertex in $\mathbb{R}^{n-k}\times C(X)$, then
 \begin{align}
 {\rm d}_{GH}(B_{p}(1),B_0(1)) <\epsilon.
 \end{align}
 \end{prop}

 To prove Proposition \ref{epsilon-regularity-1},  it suffices  to prove that ${\rm vol} (B_{p}(1))$ is close to  ${\rm vol}(B_0(1))$ according to Corollary  \ref{hausddorf-closed}.  The latter  is equivalent   to show  that ${\rm vol}(B_0(1))$ is close to  ${\rm vol}(B_{o,x}(1))$ by Remark \ref{volume-covergerce-4}.  Thus we shall estimate the volume of section  $X$.  In the following, we will use the idea in [CCT] to  turn into estimating  volume of a pre-image  of $X$ by constructing a  Gromov-Hausdorff approximation.

  Let $h_i$  $(i=1,...,n-k)$ be    $(n-k)$   $f$-harmonic  functions  on $B_p(5)$   with
appropriate boundary values as constructed in   the  proof of Splitting Theorem  \ref{splitting-theorem} (cf. Proposition \ref{proof-splitting}) and
    $h$  an  approximation  of  $\frac{r^2}{2}$  as  constructed in  the proof of  Metric Cone Theorem
   \ref{existence-metric-cone} (also  Lemma \ref{harmonic-estimate-annual-1}, Lemma \ref{harmonic-estimate-annual-2}),
 which is a solution of
  \begin{align}
  \nonumber \Delta^fh=n, ~\text{in}~B_p(5),~h|_{\partial(B_p(5))}=\frac{25}{2}.
  \end{align}
   Let
  \begin{align}
 \nonumber  w_0=2h-\Sigma h_j^2.
  \end{align}
   Define $w$ to be a solution of
 \begin{align}
  \nonumber \Delta^fw=2k,
 ~  w|\partial B_p(4)=w_0.
  \end{align}
Then $w$ is almost positive, so it can be transformed to  be positive by adding a small number.  Set
  \begin{align}
  \nonumber \mathbf{u}^2=w+\Psi>0.
  \end{align}
 We recall some estimates for functions $h_i$, $h$  and $w$:
 \begin{align}\label{orthogonal-5}
 &\frac{1}{{\rm vol}(B_p(3))}\int_{B_p(3)}\Sigma_i|\text{hess }h_i|^2+\Sigma_{i\neq j}|\langle\nabla h_i,\nabla h_j\rangle|\notag\\
 &+\frac{1}{{\rm vol}(B_p(3))}\int_{B_p(3)}\Sigma_i(|\nabla h_i|-1)^2 <\Psi,\end{align}
  \begin{align}\frac{1}{{\rm vol}(B_p(3))}\int_{B_p(3)}(|\nabla h-\nabla r|^2+|\text{hess }h-g|^2) <\Psi,
   \end{align}
    \begin{align}\label{hessian-small-4}
 \frac{1}{{\rm vol}(B_p(3))}\int_{B_p(3)}|\text{hess }{w_0}-\text{hess }w|^2 <\Psi,
 \end{align}
  and
  \begin{align}
  \frac{1}{{\rm vol}(B_p(3))}\int_{B_p(3)}|\nabla w_0-\nabla w|^2 <\Psi.
 \end{align}
 The first two  estimates  are proved in Section 2
 (cf. Lemma \ref{harmonic-estimate}, Lemma \ref{harmonic-estimate-annual-1},  Lemma \ref{harmonic-estimate-annual-2}).  We note that   the condition (\ref{cone-volume-condition-2})  in both  Lemma \ref{harmonic-estimate-annual-1} and  Lemma \ref{harmonic-estimate-annual-2} is satisfied  by (\ref{g-h-close}) according to (\ref{volume-cone}) in  Remark \ref{volume-covergerce-4}.  The others   can also be obtained  in a similar way.

  We define  maps $\Phi$ and $\Gamma$ respectively  by
   \begin{align}
   \Phi=(h_j):B_p(4)\longrightarrow \mathbb{R}^{n-k}\notag
   \end{align}
    and
\begin{align}\Gamma=(h_j,\mathbf{u}):B_p(4)\longrightarrow \mathbb{R}^{n-k+1}.\notag
   \end{align}
 Let
 \begin{align}
  \nonumber V_{\Phi,u}(z)={\rm vol}(\Phi^{-1}(z)\cap U_u),\end{align}
   where $U_u=\Gamma^{-1}(B_0^{n-k}(1)\times[0,u])$ for $u\leq 2$.
Then

\begin{lem}
 \begin{align}\label{average-volume-3}
 \frac{1}{{\rm vol}(B_0^{n-k}(1))}\int_{B_0^{n-k}(1)}|V_{\Phi,u}(z)-\frac{u^{k}}{k}{\rm vol}(X)|<\Psi.
 \end{align}
\end{lem}

\begin{proof} Set
 \begin{align}
 \nonumber v_\Phi=\nabla h_1\wedge...\wedge\nabla h_{n-k}.
 \end{align}
  Then $v_\Phi$  is     the Jacobian  of $\Phi$ in  $B_0^{n-k}(1)$.   By  (\ref{orthogonal-5}),   one can    show that  it  is almost $1$  almost everywhere in $B_0^{n-k}(1)$.
In fact,  the proof is   the same  to one of  (\ref{orthorgonal-4}).  Hence  by the coarea formula,
 we get
 \begin{align}\label{level-volume-2}
&  \frac{1}{{\rm vol}(B_0^{n-k}(1))}
 \int_{B_0^{n-k}(1)}V_{\Phi,u}(z)\notag\\
& =\frac{1}{{\rm vol}(B_0^{n-k}(1))}\int_{U_u}|v_\Phi|=\frac{{\rm vol}(U_u)}{{\rm vol}(B_0^{n-k}(1))}+\Psi.
 \end{align}

To compute the variation of $V_{\Phi,u}(z)$, we modify $V_{\Phi,u}(z)$ to
 \begin{align}
 J_{\Phi,u,\delta}=\int_{\Phi^{-1}(z)}\chi_\epsilon(|v_\Phi|^2)\psi_{u,\delta},\notag
 \end{align}
 where $\psi_{\delta,u}=\xi( \mathbf{u}^2)$ with  a cut-off function $\xi$  which satisfies
 $$\xi(t)=1,  \text{ for } t\in[0,((1-2\delta)u)^2],$$
 $$\xi(t)=0 \text{ for } t\in[((1-\delta)u)^2,u^2],
$$
and  $\chi_\epsilon(t)$ is another  cut-off function which satisfies
 $$\chi_\epsilon(t)=0, \text{ for } t\in[0,\epsilon],$$
  $$\chi_\epsilon(t)=(1-\epsilon)t, \text{ for } t\in[2\epsilon,1-\epsilon],$$
   $$\chi_\epsilon(t)=1, \text{ for } t\geq 1,$$
   $$|\chi_\epsilon'(t)|\leq 3.$$

 A direct computation shows that
 \begin{align}\label{hessian-5}
 \nonumber \frac{\partial J_{\Phi,u,\delta}}{\partial z_j}=\int_{\Phi^{-1}(z)\cap U_u}\chi_\epsilon'(|v_\Phi|^2)\sum_i a_{i,j}\nabla h_i(|v_\Phi|^2)\psi_{u,\delta},\\
 \nonumber +\int_{\Phi^{-1}(z)\cap U_u}\chi_\epsilon(|v_\Phi|^2)\sum_i a_{i,j}tr(\widehat{\text{hess } h_i})\psi_{u,\delta},\\
 +\int_{\Phi^{-1}(z)\cap U_u}\chi_\epsilon(|v_\Phi|^2)\sum_i a_{i,j}\langle\nabla \psi_\delta,\nabla h_i\rangle.
 \end{align}
 Here $a_{i,j}$ is the inverse of $\langle\nabla h_i,\nabla h_j\rangle$ so that  $\Phi_*(\sum_i a_{i,j}\nabla h_i)=\frac{\partial}{\partial z_j}$,
 and $tr(\widehat{\text{Hess }h_i})$ denotes the trace restricted to $\Phi^{-1}(z)$.
  Using the coarea formula the integrations of the first two terms at the  right side of (\ref{hessian-5})  in $B_0^{n-k}(1)$ can be controlled by  the Hessian estimate  in (\ref{orthogonal-5}).
 Moreover,  similar to  (\ref{orthorgonal-4}),  by  (\ref{orthogonal-5}) and (\ref{hessian-small-4}),  one can show,
 \begin{align}
\frac{1}{{\rm vol}(B_p(1))}\int_{B_p(1)}|\langle\nabla \mathbf{u}^2,\nabla h_j\rangle|<\Psi.\notag
 \end{align}
Thus   the integration of the third term  at the  right side of (\ref{hessian-5}) in $B^{n-k}(1)$  is also  small.
Hence we get
 \begin{align}
\frac{1}{{\rm vol}(B_0^{n-k}(1))}\int_{B_0^{n-k}(1)}|\nabla J_{\Phi,u,\delta}|< \Psi.\notag
\end{align}
 On the other hand,   by (\ref{volume-cone}),  it is easy to see
\begin{align}
\nonumber |\frac{{\rm vol}(U_u)}{{\rm vol}(B_0^{n-k}(1))}-\frac{u^k}{k}{\rm vol}(X)|<\Psi.
\end{align}
Therefore,  we  derive (\ref{average-volume-3})  from (\ref{level-volume-2}),
 \end{proof}

Similar  to  (\ref{average-volume-3}),  by  using  the  above argument   to the map  $\Gamma$,  one can  also obtain the following estimate,
 \begin{align}\label{average-volume-4}
 \frac{1}{{\rm vol}(B^{n-k}(1))\times [0,1])}\int_{B^{n-k}(1)\times [0,1]}  |V_\Gamma (z,u)-u^{k-1}{\rm vol}(X) |<\Psi,
 \end{align}
 where  $V_\Gamma (z,u)={\rm vol}(\Gamma^{-1}(z,u)).$    A similar  proof can be also found  in   Theorem 2.63 in [CCT],
  so we omit it. Thus we see

 \begin{lem}\label{level-volume}
 There exists a subset of $D_{\epsilon,l}\subseteq B^{n-k}(1)\times[0,1]$ which depending only on $\epsilon,  l$  such that
 \begin{align}
 {\rm vol}(D_{\epsilon,l}) >(1-\Psi){\rm vol}(B^{n-k}(1)\times[0,1])\end{align}
  and
  \begin{align}\label{cone-volume-3}
 |V_\Gamma (z,u)-u^{k-1}{\rm vol}(X)|<\Psi,~\forall~(z,u)\in D_{\epsilon,l}.
 \end{align}

 \end{lem}

 Next, we use the Bochner identity in terms of  Bakry-Emery Ricci curvature
  to estimate the second fundamental forms of pre-image of $\Phi,\Gamma$.
Let $v_1,v_2,...,v_m$  be $m$ smooth vector fields. Put $v=v_1\wedge v_2\wedge...\wedge v_m$. We compute
\begin{align}
\nonumber &\Delta^f|v|^2=2\langle\Delta^fv,v\rangle+2|\nabla v|^2
\end{align}
 and
\begin{align}
\nonumber &\Delta^f(|v|^2+\eta)^\frac{1}{2}=(|v|^2+\eta)^{-\frac{1}{2}}(|\nabla v|^2
-\frac{\langle\nabla v,v\rangle^2}{|v|^2+\eta})\\
&+(|v|^2+\eta)^{-\frac{1}{2}}\langle\Delta^fv,v\rangle,~\forall~\eta>0.\notag
 \end{align}
It follows,
 \begin{align}\label{second-form-related}
 \nonumber &(|v|^2+\eta)^{-\frac{1}{2}}|\pi(\nabla v)|^2\\
 &\leq-\frac{|v|}{(|v|^2+\eta)^{\frac{1}{2}}}(I-\pi)\Delta^f v+\Delta^f(|v|^2+\eta)^{\frac{1}{2}},
 \end{align}
  where  $\pi:\wedge^m\text{TM}\to v^{\perp} $  is  the compliment of  orthogonal projection to $v$.
 On the other hand,     if we choose  $v_i=\nabla l_i$ and take  map $F=(l_1,...l_m)$ and $v=v_F$, then
  \begin{align}\label{second-estimate-1}
  |v_F||\Pi_{F^{-1}(c)}|^2\leq |v_F|^{-1}|\pi(\nabla v_F)|^2,
    \end{align}
 where  $\Pi_{F^{-1}(c)}$ denote the second fundamental form of the level set $F^{-1}(c)$ in $M$.
  Hence the quantity  $(I-\pi)\Delta^f v_F$ in (\ref{second-form-related}) gives us an estimate for the second fundamental form of map $F$.

 To estimate $(I-\pi)\Delta^f v_F,$  we  use the following formula,
  \begin{align}
  \nonumber&\Delta^f \nabla l_i =\nabla \Delta^f l_i+{\rm Ric}^f(\nabla l_i,\cdot).
  \end{align}
  Note that  in our case  $\Delta^f l_i$ is constant  for map $F=\Phi$  or  $F=\Gamma=(\Phi,\mathbf{u}^2)$.    Then it is easy to see
  \begin{align}\label{second-estimate-2}
 & (I-\pi)\Delta^fv_F\notag\\
  &=2(I-\pi)(\Sigma_{j_1<j_2}\nabla l_1\wedge...\wedge\nabla_{e_s}\nabla l_{j_1}\wedge...
  \wedge\nabla_{e_s}\nabla l_{j_2}\wedge...\nabla l_m)\notag\\
  &  +{\rm tr}({\rm Ric}^f),
  \end{align}
where ${\rm tr}({\rm Ric}^f)$ is the trace over the space spanning by $\nabla l_i$.

\begin{lem}\label{small-second-form}
  There exists a subset $E_{\epsilon,l}\subseteq B^{n-k}(1)\times[0,1]$,   which depends only on $\epsilon,  l$ and satisfies
     \begin{align}
 {\rm vol}(E_{\epsilon,l})\geq(1-\Psi){\rm vol}(B^{n-k}(1)\times[0,1]), \end{align}
such that for any ($z,u)\in E_{\epsilon,l}$ it holds
   \begin{align}\label{second-flat}
 \frac{1}{V_{\Phi,u}(z) }\int_{\Phi^{-1}(z)\cap U_u}|\Pi_{\Phi^{-1}_z}|^2 <\Psi,
 \end{align}
 \begin{align}\label{second-form-1}
 \frac{1}{ V_\Gamma (z,u) }\int_{\Gamma^{-1} (z,u)}|\Pi_{\Phi^{-1}_z}|^2 <\Psi,
  \end{align}
 \begin{align}\label{second-form-2}
 \frac{1}{ V_\Gamma (z,u) }
  \int_{\Gamma^{-1} (z,u)}|\Pi_{\Gamma^{-1}(z,u)}-u^{-1}g_{\Gamma^{-1}(z,u)}\otimes\nabla u|^2 <\Psi.
  \end{align}

\end{lem}

  \begin{proof}
Let   $\phi$ be a cut-off function  with support  in $B_p(3)$  as constructed in Lemma  \ref{cut-off}.
 Note  that   $v_\Phi$  is almost $1$  almost everywhere in $U_u$ by the Hessian estimates  in (\ref{orthogonal-5}).    Then by  (\ref{second-estimate-2}),
  we  have
 \begin{align}
\nonumber  &\int_{U_u} (|v|^2+\eta)^{-\frac{1}{2}}|\pi(\nabla v)|^2
 e^{-f}d\text{v}\notag\\
&\leq \int_{B_p(3)}|v_\Phi|| (I-\pi)\Delta^fv_\Phi|e^{-f}d\text{v}+\int_{B_p(3)}\phi\Delta^f((|v_\Phi|^2+\eta)^{\frac{1}{2}}-1) e^{-f}d\text{v}\notag
 \\
 & < \Psi+\int_{B_p(3)}|\Delta^f\phi||(|v_\Phi|^2+\eta)^{\frac{1}{2}}-1|e^{-f}d\text{v}.\notag
 \end{align}
By (\ref{second-estimate-1}),  it follows
\begin{align}\label{curvature-inequality-1}
&\int_{U_u}|v_\Phi||\Pi_{\Phi^{-1}(z)}|^2e^{-f}d\text{v}\notag\\
&\le
\lim_{\eta\rightarrow 0}\int_{U_u} (|v|^2+\eta)^{-\frac{1}{2}}|\pi(\nabla v)|^2
 e^{-f}d\text{v} < \Psi.
\end{align}

On the other hand,  by  the coarea formula,  we have
   \begin{align}
\nonumber  \int_{B^{n-k}(1)}\int_{\Phi^{-1}(z)\cap U_u}|\Pi_{\Phi^{-1}(z)}|^2e^{-f}d\text{v}
=\int_{U_u}|v_\Phi||\Pi_{\Phi^{-1}(z)}|^2e^{-f}d\text{v}.
  \end{align}
  Thus (\ref{second-flat}) follows from (\ref{curvature-inequality-1})    immediately.  Again by the  coarea formula  we get (\ref{second-form-1}) from   (\ref{second-flat}).   (\ref{second-form-2}) can be also  obtained by  using the same  argument above to  the map $\Gamma$  (cf.   Theorem 3.7 in  [CCT]).

  \end{proof}

 \begin{proof}[Completion of Proof of Proposition \ref{epsilon-regularity-1}]
We will finish the proof of Proposition \ref{epsilon-regularity-1} by  applying  the Gauss-Bonnet formula to an appropriate level set of $\Gamma$.
  In  case $k=2$,   by  Lemma \ref{level-volume}, we see that there exists $(z,u)$ ($u$ is close to $1$)   such that
  \begin{align}
 \nonumber |2\pi t-\frac{1}{u} V_\Gamma (z,u)|<\Psi,\notag
 \end{align}
 where $t$ is  the radius of $X$.   Note that $X$  is a circle here.
On the other hand,  applying  the Guass-Bonnet formulam to  $\Phi^{-1}(z)\cap U_u$, we have
    \begin{align}
 \nonumber \int_{\Gamma^{-1} (z,u)}H+\int_{\Phi^{-1}(z)\cap U_u} K=2\pi\chi( \Phi^{-1}(z)\cap U_u ),
   \end{align}
where $K$ and  $H$ are Gauss curvature and mean curvature of  $\Phi^{-1}(z)\cap U_u$ and
$\Gamma^{-1} (z,u)$, respectively.
   By  (\ref{second-flat}) and (\ref{curvature-int}) together with  the Gauss-Coddazzi equation,  we  see that
  \begin{align}
\nonumber  |\int_{ \Phi^{-1}(z)\cap U_u } K| <\Psi.
  \end{align}
  Also we  get from  (\ref{second-form-2}),
  \begin{align}
\nonumber  |\int_{\Gamma^{-1} (z,u)}H -\frac{1}{u}V_\Gamma (z,u)|<\Psi.
  \end{align}
  Thus $t$ is close to $\chi( \Phi^{-1}(z)\cap U_u  )$ which is an integer.  The non-collapsing condition  implies that $\chi( \Phi^{-1}(z)\cap U_u  )$ is not zero. So $t > 1-\Psi$.   As a consequence,  the volume of ball  $B(1)\subset \mathbb{R}^{n-1}\times C(X)$  is close to one of  a unit  flat ball.    Hence by Remark \ref{volume-covergerce-4},  we see that ${\rm vol}(B_p(1))$ is close to ${\rm vol}(B_0(1))$.  Therefore,  we prove that  $B_p(1)$ is close to  $B_0(1)$ by  Corollary \ref{hausddorf-closed}.

  In case $k=3$,  we see that there exists $(z,u)$ in Lemma \ref{level-volume} such that
  \begin{align}\label{area-2}
  | V_\Gamma (z,u)-{\rm vol} (X)| <\Psi,
  \end{align}
 as  $u$  is close to 1.
  On the other hand,  by the Gauss-Bonnet formula, we have
  \begin{align}\label{gauss-bonnet}
  \int_{  \Gamma^{-1} (z,u)}K=2\pi\chi( \Gamma^{-1} (z,u)).
  \end{align}
 Since by  (\ref{second-form-1}) and (\ref{curvature-int})  together with  the  Gauss-Coddazzi equation,
\begin{align}\label{curvature-int-sub}
 \int_{  \Gamma^{-1} (z,u)}|R_{\Phi^{-1}(z)}|< \Psi,\notag
\end{align}
where $R_{\Phi^{-1}(z)}$ is the curvature tensor of the submanifold $\Phi^{-1}(z)$,
(\ref{second-form-2})  implies that
\begin{align}
\nonumber |\int_{  \Gamma^{-1} (z,u)}K-V_\Gamma (z,u)| <\Psi.
\end{align}
 By  (\ref{gauss-bonnet}), it follows
  $$ V_\Gamma (z,u) > 4\pi-\Psi,$$
    since the Euler number is even.
  Thus by (\ref{area-2}),  we get
  $${\rm vol}(B_p(1))> {\rm vol}(B(1))-\Psi > {\rm vol}(B_0(1))-\Psi.$$
 As a consequence,  the volume $B_p(1)$ is close to one of  $B_0(1)$.  Therefore,  we also  prove that $B_p(1)$ is close to $B_0(1)$ by Corollary \ref{hausddorf-closed}.

  \end{proof}

 \begin{proof}[Proof of  Theorem\ref{dimension-n-4}] First we define a  distribution  $|\widetilde{\rm Rm}|^p$  ($p\in [1,2])$   on $B_{p_\infty}(2)$ by 
$$\int_{B_{p_\infty}(2)} | \widetilde{\rm Rm}|^p~ h = \overline{ \lim_i}\int_{B_{p_i}(2)}|{\rm Rm}(g_i)|^p  h( \Psi_i(\cdot)),$$
where $\Psi_i:  B_{p_i}(2) \to B_{p_\infty}(2)$ is a sequence of Gromov-Hausdorff approximations and
$h\in C^0(B_{p_\infty}(2))$ with ${\rm supp}(h)\subset B_{p_\infty}(2)$.
Then $| \widetilde{\rm Rm}|^p $  induces a measure  $\mu$ on $B_{p_\infty}(2)$ by
$$\mu(E)=\sup_h \{\int_{B_{p_\infty}(2)} | \widetilde{\rm Rm}|^p~ h|~ 0\le h\le 1, ~h\in C^0(B_{p_\infty}(2))~{\rm and} ~{\rm supp}(h)\subset E\},$$
where $E\subset B_{p_\infty}(2)$ is any closed subset.  In particular, $\mu(B_{p_\infty}(\frac{3}{2}))<\infty$.

 Let $\epsilon$  be a small number and $\delta=\delta(\epsilon)$  the constant    determined in Proposition \ref{epsilon-regularity-1}.   Let $\delta'= (C_0^{-1}\delta(\epsilon))^p$, where the constant $C_0$  will be determined lately. 
  Define a subset in $B_{p_\infty}(2)\subset M_\infty$
  for $\theta\leq\tau$ by
  \begin{align}
  Q(\theta)=\{q\in B_{p_\infty}(1)|~\frac{  \mu (B_q(s))}{\text{vol}(B_q(s))} \geq\delta' s^{-2p}, ~ \exists ~s\leq \theta \}.
  \end{align}
 We prove

  \begin{claim}\label{claim-sing-set}
\begin{align}
B_{p_\infty}(1)\subseteq \mathcal{R}_{3\epsilon}\cup Q(\theta)\cup \mathcal{S}_{n-4}.
\end{align}
  \end{claim}

  Suppose that the claim is not true.  Then there exist a point $z\bar
{\in}\mathcal{R}_{3\epsilon}\cup Q(\theta)\cup \mathcal{S}_{n-4}$ and a tangent cone $T_zY$  which is $\mathbb{R}^{n-k}\times C(X)$ for $k=2$ or $3$ and ${\rm d}_{GH}(B_{z_\infty}(1),B_0(1))>3\epsilon$, where $z_\infty\cong z$. Thus there is a sequence $r_i$ approaching $0$ such that $(Y,\frac{d}{r_i};z)\rightarrow T_zY$ in  the Gromov-Hausdorff
  topology.
 Hence for  large enough $i$,  we have
\begin{align}
{\rm d}_{GH}(B_z(r_i),B_0(r_i))\geq 3\epsilon r_i, \notag\end{align}
 \begin{align} {\rm d}_{GH}(B_z(lr_i),B_{(0,x)}(lr_i))\leq  \frac{1}{2} r_i\eta, \notag
\end{align}
and 
\begin{align} \frac{  \mu (B_z(4r_i))}{\text{vol}(B_z(4r_i))} < \delta' (4r_i)^{-2p},\notag
\end{align}
where   $\eta=\eta(\epsilon)<<1$  and $l=l(\epsilon)>>1$  are both determined  in Proposition \ref{epsilon-regularity-1}.  For fixed  $i$ in the above inequalities,  we take   $j$ large enough  and  choose a point $z_j\in M_j\to z$   such that
\begin{align}\label{equ-con-1}
{\rm d}_{GH}(B_{z_j}(r_i),B_0(r_i))\geq2\epsilon r_i,\end{align}
\begin{align}\label{l-ball-closed} {\rm d}_{GH}(B_{z_j}(lr_i),B_{(0,x)}(lr_i))\leq r_i\eta, 
\end{align}
and 
\begin{align}\label{small-p-energy} \frac{1}{\text{vol}(B_{z_j}(4r_i))} \int_{B_{z_j}(3r_i)}|{\rm Rm}(g_i)|^p  < 2 \delta' (4r_i)^{-2p}.
\end{align}
By the volume comparison, we get from (\ref{small-p-energy}),  
\begin{align}
 \frac{1}{\text{vol}(B_{z_j}(3r_i))} \int_{B_{z_j}(3r_i)}|{\rm Rm}(g_i)|^p  < C_0' \delta' (3r_i)^{-2p},\notag
\end{align}
where  $C_0=C_0(A,v,\Lambda)$.
The H\"older inequality implies
\begin{align}\label{curvaure-comparison}  
 \frac{1}{\text{vol}(B_{z_j}(3r_i))} \int_{B_{z_j}(3r_i)}|{\rm Rm}(g_i)|  < C_0 ( \delta' )^{\frac{1}{p}}(3r_i)^{-2}=\delta(\epsilon) (3r_i)^{-2}.
\end{align}
Thus applying Proposition  \ref{epsilon-regularity-1} to the manifold $M_j$ with   rescaling  metric $\frac{g_i}{ r_i}$   together with  conditions   (\ref{equ-con-1}), (\ref{l-ball-closed}) and  (\ref{curvaure-comparison} ),     we obtain
\begin{align}{\rm d}_{GH}(B_{z_j}(r_i),B_0(r_i))\leq\epsilon r_i.\notag
\end{align}
But this is  impossible by  (\ref{equ-con-1}).  The claim is proved.

 By Claim \ref{claim-sing-set},  $B_{p_\infty}(1)\setminus \mathcal{R}_{2\epsilon}\subseteq Q(\theta)\cup \mathcal{S}_{n-4}$. Now we  estimate $\mathcal{H}^{n-2p}(Q(\theta))$.  By Vitali covering lemma,  for any $r>0$ there is a collection of disjoint balls $B_{q_j}(s_j) \subset Q(\theta)$ ($s_j\le r)$  such that $\bigcup B_{q_j}(5s_j)\supseteq Q(\theta)$   with property 
$$\frac{1}{{\rm vol}(B_{q_j}(s_j))} \mu(B_{q_j}(s_j))  \geq\delta' s_j^{-2p}, $$
where $q_j\in Q(\theta)$.     By the volume comparison,  it follows
\begin{align}
\Sigma s_j^{n-2p}\leq \frac{c(\Lambda,v,A)  \mu(B_{p_\infty}(\frac{3}{2}))}{\delta}.
\end{align}
 Taking $ r\to 0$, we  get
  \begin{align}  \mathcal{H}^{n-2p}(Q(\theta))<\infty.\notag
\end{align}
 Hence (\ref{p-les-2})and (\ref{p=2}) follows from the above estimate immediately.  

  \end{proof}

\vskip3mm

  \section{Structure of singular set II:  Case of K\"ahler metrics}

In this section,  we  study the limit  space of  a sequence of  K\"ahler
  metrics arising from  solutions of certain complex Monge-Amp\`ere equations  for the existence 
 of K\"ahler-Ricci soliton  on  a Fano manifold  via  the continuity method [TZ1], [TZ2]. We assume that $(M,g)$ is a compact K\"ahler manifold with positive first Chern class $c_1(M)>0$ (namely, $M$ is Fano),  and
$\omega_g$ is the   K\"ahler form  of $g$  in $2\pi c_1(M)$.  Then  there exists
  a Ricci potential $h$ of the  metric $g$ such that
 \begin{align}{\rm Ric}(g)-\omega_g=\sqrt{-1}\partial\overline\partial h,~
 \int_M e^h\omega_g^n=\int_M\omega^n=V.\notag
 \end{align}

 In [TZ1], Tian and Zhu considered a family of complex Monge-Amp\`ere equations  for K\"ahler potentials $\phi$ on $M$,
 \begin{align}\label{ma-equ} \text{det}(g_{i\overline j}+\phi_{i\overline j})=\text{det}(g_{i\overline j})e^{h-\theta_X-X(\phi)-t\phi},
 \end{align}
 where $t\in [0,1]$ is a parameter and $\theta_X$ is a real valued  potential of a reductive  holomorphic vector field on $M$ which is defined
\begin{align}\bar{\partial}\theta_X=i_X\omega_{g}, ~
 \int_M e^{\theta_X}\omega_g^n=V,\notag
 \end{align}
according to the choice of $g$ with $K_X$-invariant.  The equations (\ref{ma-equ}) are equal  to
 \begin{align}\label{ricci-equ}
 {\rm Ric}(\omega_{\phi})-L_X\omega_{\phi}=t\omega_{\phi}+(1-t)\omega_g.
 \end{align}
Thus $\omega_{\phi}$  will define a K\"ahler-Ricci soliton if $\phi$  is a solution of  (\ref{ma-equ}) at $t=1$. It was proved
the set $I$ of $t$ for which (\ref{ma-equ}) is solvable is open [TZ1]. In the other words, there exists $T\le 1$ such that
 $I=[0,T)$. (\ref{ricci-equ}) implies
 \begin{align}\label{curvature-condition-kahler}
{\rm Ric}(\omega_{\phi})+\sqrt{-1}\partial\overline\partial(-\theta_X(\phi))\ge t\omega_{\phi},\end{align}
where $\theta_X(\phi)=\theta_X +X(\phi)$ is a potential of $X$ associated to $\omega_\phi$, which is uniformly bounded [Zh].

 \begin{lem}\label{lemma-partial-theta}
   $|\overline\partial(\theta_X+X(\phi))|=|X|_{\omega_{\phi}}$ and  $\Delta_{\overline\partial}(\theta_X(\phi))$  are   both uniformly  bound by  $C(M,\omega, X)$,
 where  $\Delta_{\overline\partial}=\frac{1}{2}\Delta$ is a $\overline\partial$-Lapalace operator associated to
$\omega_{\phi}$.

 \end{lem}

 \begin{proof} We will use the maximum principle to prove the lemma. First we recall that
$\theta_X(\phi)$ satisfies an identity [Fu],
$$\Delta_{\overline\partial}[\theta_X(\phi)]+\theta_X(\phi)+X(h)=0,$$
 where $
 h$ is a Ricci potential of  K\"ahler form $\omega_{\phi}$ at $t$.  Note that
$$h =\theta_X(\phi)+(t-1)\phi$$
by (\ref{ricci-equ}).
 Thus $\theta_X(\phi)$ satisfies
 \begin{align}\label{lapalace-theta}
 \Delta_{\overline\partial}[\theta_X(\phi)]+ |\overline\partial \theta_X(\phi)|^2+\theta_X(\phi)=(1-t)X(\phi).
 \end{align}

By  the Bochner formula, one sees
 \begin{align}
 &\Delta_{\overline\partial} (|\overline\partial \theta_X(\phi)|^2)\notag\\
 &=|\nabla\overline\nabla\theta_X(\phi)|^2+2\text{re}(
 \langle\overline\partial \theta_X(\phi),\overline\partial\Delta_{\overline\partial}\theta_X(\phi)\rangle)+{\rm Ric}(\overline\partial \theta_X(\phi),\overline\partial \theta_X(\phi))\notag
 \end{align}
It follows
\begin{align}
&(\Delta_{\overline\partial} +X)(|\overline\partial \theta_X(\phi)|^2)\notag\\
&=|\nabla\overline\nabla \theta_X(\phi)|^2+2\text{re}(
 \langle\overline\partial \theta_X(\phi),\overline\partial(\Delta_{\overline\partial}\theta_X(\phi)+|\overline\partial \theta_X(\phi)|^2)\rangle)\notag\\
 &+ ({\rm Ric}-\nabla\overline\nabla\theta_X(\phi))(\overline\partial \theta_X(\phi),
 \overline\partial \theta_X(\phi)).\notag\end{align}
Thus by  (\ref{lapalace-theta}),  we get
\begin{align}\label{lapalace-partial-theta}
(\Delta_{\overline\partial} +X)(|\overline\partial \theta_X(\phi)|^2)= |\nabla\overline\nabla \theta_X(\phi)|^2-t|\overline\partial \theta_X(\phi)|^2
-(1-t)|X|_{g}^2.
 \end{align}
 Note that
 $$ |\nabla\overline\nabla \theta_X(\phi)|^2\ge \frac{(\Delta_{\overline\partial}\theta_X(\phi))^2}{n}
 \ge \frac{(|\overline\partial\theta_X(\phi)|^2-C_1)^2}{n},
 $$
 where $C_1=\max_M\{|\theta_X(\phi)-(1-t)X(\phi)|\}.$
 Apply the maximum principle to $|\overline\partial \theta_X(\phi)|^2$ in (\ref{lapalace-partial-theta}),
  we derive at a maximal point of $|\overline\partial \theta_X(\phi)|^2$,
 \begin{align} 0\geq \frac{1}{n}(|\overline\partial\theta_X(\phi)|^2-C_1)^2
 -t|\overline\partial\theta_X(\phi)|^2-C_2.
  \end{align}
  Therefore,   the gradient  estimate  of $\theta_X(\phi)$  follows from the above inequality immediately.  By
 (\ref{lapalace-theta}), we also get the $\overline\partial$-Lapalace estimate of $\theta_X(\phi)$.

\end{proof}

By Lemma \ref{lemma-partial-theta} and  Theorem \ref{dimension-k}, we prove

\begin{theo}\label{thm-kahler-1}
 For any sequence of K\"ahler metrics $g_{t_i}$ associated to solutions $\phi_{t_i}$ of equations (\ref{ma-equ}) at $t=t_i\in I$,
 there exists a subsequence which converge to  a  limit metric space $Y$  in   the Gromov-Hausdorff topology. Moreover,
  $\mathcal{S}(Y)=\mathcal{S}_{2n-2}$.  In particular, the complex codimension of singularities of $Y$ is at least 1.
\end{theo}

\begin{proof} We suffice to verify that
 \begin{align}\label{volume-lower-bound}{\rm vol}_{g_t}(B_{p}(1))\geq v>0.~\forall~p\in M.\end{align}
But this is just a consequence of application of Volume Comparison Theorem \ref{volume-comparison} since
the diameter of $g_t$ is uniformly bounded by a result of Mabuchi [Ma].
\end{proof}

 In a special case $t_i\to 1$ when $I=[0,1)$ in  Theorem \ref{thm-kahler-1},  we  can strengthen Theorem \ref{thm-kahler-1} as follows.

\begin{theo}\label{thm-kahler-2}
Let $g_{t_i}$ be a  sequence of K\"ahler metrics  in  Theorem \ref{thm-kahler-1} with $t_i\to 1$.
 Then  $\mathcal{S}(Y)=\mathcal{S}_{2n-4}$.  In particular, the complex codimension of singularities of  $Y$ is at least 2.
\end{theo}

$I=[0,1)$ can be guaranteed  when the modified Mabuchi $K$-energy is bounded  below  and $X$ is  a soliton  holomorphic  vector field which determined by the modified Futaki-invariant [TZ2].
This can be proved following  an argument  by Futaki for the study of almost K\"ahler-Einstein metric under an assumption  that   the Mabuchi $K$-energy is bounded  below on a Fano manifold [Fu]. Thus as a corollary of  Theorem \ref{thm-kahler-2}, we have

\begin{cor} Suppose that the modified $K$-energy is bounded  below on a Fano manifold. There exists a  subsequence of    weak  almost K\"ahler-Ricci solitons on $M$ which converge to  a  limit metric space $Y$  in   the Gromov-Hausdorff topology.  Moreover, the complex codimension of singularities of $Y$ is at least 2.

\end{cor}

\begin{rem}In case that  $X=0$,  the  modified Mabuchi $K$-energy is just   the  Mabuchi $K$-energy.   In this case,  the $K$-energy is
bounded from below  is equivalent to that  the Fano manifold is $K$-semistable  by  a recent work  of  Li [Li].

\end{rem}

It is  useful to  introduce a more general  sequence  of K\"ahler metrics than one in Theorem \ref{thm-kahler-2} inspired by a recent work of Wang and Tian [WT].

\begin{defi}\label{almost-kr-soliton}   We call a sequence of  K\"ahler metrics $(M_i, J_i,  g_i)$
  weak  almost K\"ahler-Ricci solitons if  there are uniform constants $\Lambda$ and $A$  such that
\begin{align}
& i)~  {\rm Ric}(g_i)+ \nabla\overline\nabla  f_i\ge -\Lambda^2 g_i,~\nabla\nabla f_i=0;\notag\\
&ii)~ \|\overline\partial f_i\|_{g_i}\le A;\notag\\
&iii) ~\lim_{i\to\infty}\|{\rm Ric}(g_i)-g_i+\nabla\overline\nabla f_i\|_{L^1(g_i)}= 0.\notag
\end{align}
Here  $f_i$ are  some smooth functions and $\bar\partial f_i$  define  reductive holomorphic vector fields
 on  Fano manifolds $(M_i,J_i)$.
 \end{defi}

\begin{lem}\label{condtion-almost-kr-soliton}  Let $\{g_{t_i}\}$ be a  sequence of K\"ahler metrics  in  Theorem \ref{thm-kahler-1} with $t_i\to 1$.
Then $\{g_{t_i}\}$ is a sequence of weak  almost K\"ahler-Ricci solitons on $M$.
\end{lem}

\begin{proof} By  Lemma \ref{lemma-partial-theta}, it suffice to check the condition iii) in  Definition \ref{almost-kr-soliton}.
In fact, we have
\begin{align}
&\int_M|{\rm Ric}(\omega_\phi)-\sqrt{-1}\partial\bar{\partial}\theta_X(\phi)-\omega_\phi|\notag\\
&\le\int_M|{\rm Ric}(\omega_\phi)-\sqrt{-1}\partial\bar{\partial}\theta_X(\phi)-t\omega_\phi |+n(1-t){\rm vol}(M)\notag\\
&=\int_M({\rm Ric}(\omega_\phi)-\sqrt{-1}\partial\bar{\partial}\theta_X(\phi)-t\omega_\phi)\wedge
\frac{\omega_\phi^{n-1}}{(n-1)!}+n(1-t){\rm vol}(M)\notag\\
&=2n(1-t)\text{Vol}(M)\to 0\notag.
\end{align}

\end{proof}

We now begin to prove Theorem  \ref{thm-kahler-3}.  As in the proof of Theorem \ref{dimension-n-4}.
We need the following  $\epsilon$-regularity result  for the tangent cone.

 \begin{lem}\label{epsilon-regularity}
For any $\mu_0,\epsilon>0$, there exist  small  numbers $\delta=\delta(v,\epsilon,n)$, $\eta=\eta (v,\epsilon, n)$
, $\tau=\tau(v,\epsilon,n)$ and a big number $l=l(v,\epsilon,n)$  such that if a K\"{a}hler manifold $(M^n,g)$ satisfies
\begin{align}
&i)~ {\rm Ric}^f_M(g) >-(n-1)\tau^2 g, \nabla\nabla f=0,\notag\\
&ii)~ {\rm vol}_g(B_p(1))\geq \mu_0,\notag\\
&iii)~ |\nabla f|< \tau,\notag\\
&iv)~  \frac{1}{{\rm vol}(B_p(2))}\int_{B_p(2)}|{\rm Ric}(g)+\nabla\overline\nabla f|dV_g <\delta,\notag\\
&v)~ {\rm d}_{GH}(B_p(l),B_{(0,x)}(l))<\eta,\notag
\end{align}
where $B_{(0,x)}(l)$ is a  $l$-radius ball in  cone $\mathbb{R}^{2n-2}\times C(X)$ centered at the  vertex $(0,x)$
 for some metric space  $X$,  then
 \begin{align}
 {\rm d}_{GH}(B_p(1),B(1))<\epsilon.
 \end{align}
 \end{lem}

\begin{proof}
  The proof  of  Lemma \ref{epsilon-regularity} is a modification to  one of  Proposition \ref{epsilon-regularity-1}. Note that   $X$ is a circle of radius $t$ in present case.   It  suffices to  show that
  $t$ is close to  $2\pi$ by  Lemma \ref{hausddorf-closed}.
   Let   $ \Phi=(h_1,...,h_{2n-2})$  and  $\Gamma=(\Phi, \mathbf{u})$ be two maps constructed  in Proposition \ref{epsilon-regularity-1}.  By  Proposition \ref{J-invariant-property}  in Appendix 2, we may also assume
      \begin{align}\label{J-invariant}
 \int_{B_{p}(3)}|\nabla h_{n-1+i}-\mathbf{J}\nabla h_{i}|^2< \Psi(\tau, \epsilon,\frac{1}{l};v).
 \end{align}

We shall compute  the differential characteristic $\widehat{c_{1,\nabla}}$ of  tangent bundle $(TM,$ $\nabla)$
restricted on  $\Gamma^{-1}(z,u)=\Phi^{-1}(z) \cap U_u$ with fixed $z$ (cf. [Ch3]),   where $\nabla$ is the Levi-Civita connection on $TM$ and $(z,u)$ is a regular point of $\Gamma$ such that both Lemma \ref{level-volume}
and  Lemma \ref{small-second-form} hold.  It is easy to see that by the coarea formula and the condition iv),  the set
 \begin{align}
 \nonumber D=\{ z|~&  \Phi^{-1}(z)\cap U_u ~\text {is a regular  surface in  }~M~\text{and}\notag\\
 &\int_{\Phi^{-1}(z)\cap U_u}|{\rm Ric}(g)+\nabla\overline\nabla f| <c\delta\}
 \end{align}
 has a  positive volume in $\mathbb R^{2n-2}$  for some  constant $c$ which depends  only on $n$.

 For each  $z\in  D$,  we have the estimate
  \begin{align}\label{small-class-1}
  &|\int_{\Phi^{-1}(z)\cap U_u}\text{Ric }(\omega_g)|\notag\\
  &\le \int_{\Phi^{-1}(z)\cap U_u} |{\rm Ric}(g)+\nabla\overline\nabla f|
    +|\int_{\Phi^{-1}(z)\cap U_u} \sqrt{-1} \partial\overline\partial f |\notag \\
  &\leq c\delta+\int_{ \Gamma^{-1}(z,u)}|\nabla f|\leq c\delta+{\rm vol}(\Gamma^{-1}(z,u))\tau.
  \end{align}
Since
   \begin{align}
  \int_{ \Gamma^{-1}(z,u)}\widehat{c_{1,\nabla}}=\int_{\Phi^{-1}(z)\cap U_u}\text{Ric }(\omega_g), ~ \text{mod } \mathbb{Z},\notag
  \end{align}
 we get
 \begin{align}\label{small-class-2}
  \int_{ \Gamma^{-1}(z,u)}\widehat{c_{1,\nabla}}=\Psi,  ~ \text{mod } \mathbb{Z}.
  \end{align}
To compute the left term of (\ref{small-class-2}),
 we  will decompose  the tangent bundle $(TM,$ $\nabla)$ over $\Gamma^{-1}(z,u)$ as follows.

 By our construction of the map $\Gamma$,   using the coarea formula, we may   assume  that
  \begin{align}
  & i)~ \int_{\Gamma^{-1}(z,u) }|\langle\nabla h_i,\nabla h_j\rangle-\delta_{ij}|< \Psi, \notag\\
  &ii)~\int_{\Gamma^{-1}(z,u)}|\text{hess }h_i|< \Psi, \notag\\
  &iii)~\int_{\Gamma^{-1}(z,u) }|\langle\nabla \mathbf u^2,\nabla h_j\rangle| < \Psi, \notag\\
  &iv)\int_{  \Gamma^{-1}(z,u)}|\nabla\langle\nabla \mathbf u^2,\nabla h_j\rangle|< \Psi.\notag
  \end{align}
  Since $\Gamma^{-1}(z,u)$ is one dimensional manifold with bounded length,  the conditions i- ii) and iii-iv) imply
  $$|\langle\nabla h_i,\nabla h_j\rangle-\delta_{ij}|~\text{and} ~|\langle\nabla \mathbf u^2,\nabla h_j\rangle|$$
  are both small on   $\Gamma^{-1}(z,u)$,  respectively.  Moreover,  applying  the coarea formula to (\ref{J-invariant}) together with the above condition ii),  we also get
  \begin{align}
 \nonumber |\nabla h_{n-1+i}- \mathbf J\nabla h_i|< \Psi.
  \end{align}
  Hence by using  the Gram-Schmidt process, we obtain  $(2n-1)$   orthogonal   sections of $TM$ over  $\Gamma^{-1}(z,u)$,
  $$e_i,\mathbf{J}(e_i)~(1\leq i\leq n-1), \mathbf{N}$$
     from   sections  $\nabla h_i$ $(1\leq i\leq n-1)$, $\nabla \mathbf{u}$.   Denote  $\mathbb{E}$ to be  the sub-bundle spanning by $e_i,\mathbf{J}(e_i)$ and decompose  $TM$ into
    \begin{align}
    TM=\mathbb{E}\oplus\mathbb{E}^\perp
    \end{align}
 where    $\mathbb{E}^\perp$ is the orthogonal complement of $\mathbb{E}$.   We  introduce a  Whitney sum  connection  $\nabla'$  on  $TM$  over $\Gamma^{-1}(z,u)$  by  combining two  projection connections on   $\mathbb{E}$  and $\mathbb{E}^\perp$, which are both
 induced by  $\nabla$.   Then  by the condition ii), it is easy to show
    \begin{align}\label{small -connection-1}
    \int_{\Gamma^{-1}(z,u)}|\nabla-\nabla'|< \Psi,
    \end{align}
   where  $\nabla-\nabla'$  is  regarded as a 1-form on $\text{End}(TM)$.   Also we can  introduce  another connection $\nabla''$
   which is flat on $\mathbb{E}$.   Namely,  $\nabla''$ satisfies
     $$\nabla''(e_i)=\nabla''(\mathbf{J}(e_i))=0.$$
    Similar to (\ref{small -connection-1}),   we have
     \begin{align}\label{small -connection-2}
    \int_{\Gamma^{-1}(z,u)}|\nabla''-\nabla'|< \Psi.
    \end{align}
    Therefore, combining (\ref{small -connection-1}) and (\ref{small -connection-2}), we derive
     $$|(\widehat{c_{1,\nabla''}}-\widehat{c_{1,\nabla}})(\Gamma^{-1}(z,u))|<<1.$$

  On  the other hand,  by the flatness of  $\nabla''$ on $\mathbb{E}$ over $\Gamma^{-1}(z,u)$,  the quantity  $2\pi \widehat{c_{1,\nabla''}} (\Gamma^{-1}(z,u))$ is  just equal to the holonomy of the connection around $\Gamma^{-1}(z,u)$ (measured by  angle),
  \begin{align}
  2\pi \widehat{c_{1,\nabla''}} (\Gamma^{-1}(z,u))=\int_{\Gamma^{-1}(z,u)}\langle \nabla''_X \mathbf{N}, \mathbf J\mathbf N\rangle,
  \end{align}
  where $X$ is the unit tangent vector of $\Gamma^{-1}(z,u)$.  Thus by the choice of  $\mathbf{N}$ together with  (\ref{small -connection-1}), (\ref {small -connection-2})  and (\ref{second-form-2}),  we see that the angle is close to the length of $\Gamma^{-1}(z,u)$.   By (\ref{small-class-2}),   it follows that  $\frac{{\rm vol}(\Gamma^{-1}(z,u))}{2\pi}$ is close to zero modulo integers.  Hence, the  non-collapsing  of $ B_{(0,x)}(1)$ implies that ${\rm vol}(\Gamma^{-1}(z,u))$ is close to $2\pi$.     Consequently,  we prove that  $t$ is close
   to  $2\pi$ by (\ref{cone-volume-3}) in Lemma \ref{level-volume}.

  \end{proof}

\begin{proof}[Proof of  Theorem \ref{thm-kahler-3}]  By  Volume Comparison Theorem \ref {volume-comparison}, for any
$r\le 1$,  we have
$${\rm vol}_{g_i}({\rm vol}(B_p(r))\ge \lambda_0r^n,~\forall ~p\in ~M_i,$$
where $\lambda_0$ depends only on the constants $\Lambda, A, v$ in  Definition \ref{almost-kr-soliton}.  Thus by Gromov's  compactness theorem [Gr],  there exists  a subsequence of $(M_i,g_i;p_i)$  which converge to  a   metric space $Y_\infty$  in  the  pointed Gromov-Hausdorff topology.  In the remaining, we show that $\mathcal{S}(Y_\infty)=\mathcal{S}_{2n-4}$.
We will use the  argument by contradiction.  On the contrary,  for a ball $B_y(1)\subset Y$,  by Proposition \ref{prop-even-dim} in Appendix 2,  there exists a  point
 $z\in S\cap B_y(1)\nsubseteqq S_{2n-4}$
and  there exists a sequence $\{r_i\}~(r_i\to 0)$ such that  $(Y,\frac{d}{r_i^2};z)$ converge a tangent cone $T_zY=\mathbb{R}^{2n-2}\times C(X)$. This implies that exists an $\epsilon>0$ such that the unit metric ball  $B_{z_\infty}(1)\subset T_zY$ centered at $z_\infty\cong z$ satisfies
\begin{align} {\rm d}_{GH}(B_{z_\infty}(1),B(1))>2\epsilon,\end{align}
and for any $l>>1$ and $\epsilon<<1$ one can  choose sufficiently   large numbers  $i$ and $k$ such that
\begin{align}\label{assumption-thm-kahler-3}
&{\rm d}_{GH}(\hat B_{z_k}(1),B(1))>\epsilon,\\
&{\rm d}_{GH}(\hat B_{z_k}(l), B_{(0,x)}(l))< \eta\notag,
\end{align}
where $z_k\in M_k\to z\in Y$ as $k\to \infty$, and $ \hat B_{z_k}(1)$ and $\hat B_{z_k}(l)$ are two balls with radius $1$ and $l$ respectively in $(M_k,\frac{g_k}{r_i^2})=( M_k,\hat g_k)$ .
On the other hand, by using  Volume Comparison Theorem \ref{volume-comparison},  for fixed  $i$, we can choose large enough $k$ such that
\begin{align}
\frac{r_i^2}{{\rm vol}( B_{z_k}(2r_i))}\int_{ B_{z_k}(2r_i)}|{\rm Ric}(g_k)-g_k+\nabla\overline\nabla f_k| d\text{v}_{g_k}<\frac{1}{2} \delta.\notag
\end{align}
Since
\begin{align}
\frac{r_i^2}{{\rm vol}( B_{z_k}(2r_i))}\int_{ B_{z_k}(2r_i)}|g_k|d\text{v}_{g_k}\leq c(n,C)r_i^2\notag,
\end{align}
we have
\begin{align}
\frac{1}{{\rm vol}( \hat B_{z_k}(2))}\int_{\hat B_{z_k}(2)}|{\rm Ric}(\hat g_k)+\nabla\overline\nabla f_k| d\text{v}_{\hat g_k}< \delta.
\end{align}
Hence, for large $k$, $( M_k,\hat g_k)$ satisfies the conditions i-v) in Lemma \ref{epsilon-regularity}, and consequently, we get
$${\rm d}_{GH}(\hat B_{z_k}(1),B(1))< \epsilon,$$
which is a contradiction to (\ref{assumption-thm-kahler-3}).  The theorem is proved.
\end{proof}

 Theorem \ref{thm-kahler-2}  follows from  Theorem \ref{thm-kahler-3}
with the help of Lemma \ref{condtion-almost-kr-soliton} and the relation (\ref{volume-lower-bound}).

\vskip3mm

\section{Appendix 1}

This appendix is a discussion  about how to use the technique of  conformal transformation  from  [TZh] to prove  Theorem \ref {thm-kahler-1} and Theorem \ref {thm-kahler-2} in Section 6.  We would like to emphasis on the different  situation  after the change of Ricci curvature by the  conformal transformation.

First, Theorem  \ref {thm-kahler-1} can be proved by using  the conformal  technique.   In fact,
by the formula of Ricci curvature for conformal metric  $e^{2u}g$,
\begin{align}\label{conformal-curvature}
&\text{Ric }(e^{2u}g)\notag\\
&=\text{Ric }(g)-(n-2)(\text{hess }u-du\otimes du)+(\Delta u+(n-2)|\nabla u|^2)g,
\end{align}
the condition  $\text{Ric }^f_M(g)\geq -C$  implies that Ricci curvature  $\text{Ric }(e^{-\frac{2f}{n-2}}g)$ of  conformal  metric $e^{-\frac{2f}{n-2}}g$  is bounded
 below if  both $\nabla f$ and $\Delta f$ are bounded.  Thus by  Lemma \ref{lemma-partial-theta},   we see that
$$\text{Ric }(e^{\frac{2\theta_X(\phi_t)}{n-2}}g_{t})$$
is uniformly  bounded  below.  Hence,   Theorem  \ref {thm-kahler-1}  follows from Theorem 6.2 in  [CC2]  immediately.

Secondly,  following the proof of Theorem  5.4  in  [Ch3],  Lemma \ref{epsilon-regularity} with an additional condition
$ \text{vi)}~ |\Delta f|<\tau$
 can be proved   by using the conformal change of the bundle metric.   We note that the condition vi)  can be guaranteed  for      the K\"ahler manifolds  $(M,g_t)$  in Theorem  \ref {thm-kahler-2} with  blowing-up metrics.  Thus by   (\ref{conformal-curvature}),
the  Ricci curvature of  blowing-up metric of  $e^{\frac{2\theta_X(\phi_t)}{n-2}}g_{t}$  is almost positive.

 For a K\"{a}hler manifold $(M,g,\mathbf{J})$, the $(1,0)$-type Hermitian connection $\nabla$ on the holomorphic bundle $(TM,h)$ is same as the  Levi-Civita connection, where $h$ is the Hermitian metric corresponding to $g$.  Then  $c_{1,\nabla}$  of  $(TM,h)$  is the same as the Ricci form of $g$.  If  we choose  a Hermitian metric $e^{\psi}g$  for a smooth function $\psi$,   then
 $$\tilde{\nabla}=\nabla+\partial \psi$$
  is the corresponding $(1,0)$-type Hermitian connection.  It follows
\begin{align}F^{\tilde{\nabla}}=F^\nabla+d\partial \psi\notag\end{align}
and
\begin{align}\label{equ-conf}
 \sqrt{-1}tr(F^{\tilde{\nabla}})=\sqrt{-1}tr(F^\nabla)-n\sqrt{-1}\partial\bar{\partial}\psi,
\end{align}
where $F^\nabla$ ($F^{\tilde{\nabla}}$) denotes the curvature of  the connection $\nabla$ ($\tilde \nabla$) on $TM$.   Thus  by  putting $\psi=-\frac{2\pi}{n}f$ and
using (\ref{equ-conf}),  we have
\begin{align}\label{modified-small-connection}
\widehat{c_{1,\tilde{\nabla}}}( \Gamma^{-1}(z,u) )=\int_{  \Gamma^{-1}(z,u)}|\text{Ric }(\omega_g)+\sqrt{-1}\partial\bar{\partial}f|, ~ \text{mod } \mathbb{Z},
\end{align}
where the map $\Gamma$  is defined as in Section 5 and Section 6  for the conformal metric $\tilde g=e^{-\frac{2f}{n-2}}g$.  Thus $\widehat{c_{1,\tilde{\nabla}}}(\Gamma^{-1}(z,u))$ is small modulo integers.   Moreover,  by  Theorem 3.7 in [CCT] (compared to Lemma  \ref{small-second-form}  in Section 5) ,  it holds
 \begin{align}\label{second-form-conformal-metric}
 \frac{1}{ V_\Gamma (z,u) }
  \int_{\Gamma^{-1} (z,u)}|\Pi_{\Gamma^{-1}(z,u)}-u^{-1}\tilde g_{\Gamma^{-1}(z,u)}\otimes\nabla u|^2<\Psi.  \end{align}

 On the other hand,     since  the Ricci curvature of $\tilde g$  is
  almost positive,    for the connection $\tilde\nabla$,  we can follow the argument in  proof of Theorem 5.4
  [Ch3]  to show that
 the quantity  $2\pi \widehat{c_{1,\tilde\nabla}} (\Gamma^{-1}(z,u))$ is  close to a  holonomy of  another perturbation  connection  $\tilde\nabla''$ of $\tilde\nabla$ around $\Gamma^{-1}(z,u)$  (also see the argument in  proof of  Lemma  \ref{epsilon-regularity}). The late is close to
 $$\int_{\Gamma^{-1} (z,u)}\Pi_{\Gamma^{-1}(z,u)}.$$
Thus combining  (\ref{modified-small-connection})  and  (\ref{second-form-conformal-metric}), we get
 $$
 |\widehat{c_{1,\tilde\nabla}}(\Gamma^{-1}(z,u))-\frac{{\rm vol}(\Gamma^{-1}(z,u))}{2\pi}|<\Psi.$$
It  follows  that  the diameter of section $X$ in two dimensional cone $C(X)$ with rescaled  cone metric is close to $2\pi$.
Thus the Gromov-Hausdorff distance between $B_p(1)$ and $B_{(0,x)}(1)$ both with rescaled metrics  is close to zero.
 By Theorem 9.69 in [Co3], we prove  Lemma \ref{epsilon-regularity} with the additional  condition vi).
 Theorem \ref {thm-kahler-2}  follows  from applying   Lemma \ref{epsilon-regularity}  to the sequence $\{(M,g_t)\}$ $(t\to 1)$ with  blowing-up metrics,  for details to see  the proof of  Theorem \ref {thm-kahler-3} in the end of Section 6.

\vskip3mm

\section{Appendix 2}

 In this appendix, we prove  (\ref{J-invariant}) in Section 6.   We need several lemmas.  First,
as an  application of  Lemma \ref{harmonic-estimate-annual-2},  we have

\begin{lem} \label{almost-gradient-vector}
Under the conditions of Lemma \ref{harmonic-estimate-annual-1}, for a vector field $X$ on $A_p(a,b)$  which satisfies
\begin{align}\label{conditions}
|X|_{C^0(A_p(a,b))}\leq  D , \frac{1}{{\rm vol }^f(A_p(a,b))}\int_{A_p(a,b)}|\nabla X|^2 d\text{v}^f< \delta,
\end{align}
there exists a $f$-harmonic function $\theta$ defined in $A_p(a_2,b_2)$ such that
\begin{align}\label{gradient-est-app}
\frac{1}{{\rm vol }^f(A_p(a_2,b_2))}\int_{A_p(a_2,b_2)}|\nabla\theta-X|^2 d\text{v}^f
< \Psi(\epsilon,\omega, \delta;A,a_1,b_1,a_2,a,b),\end{align}
and
\begin{align}&\frac{1}{{\rm vol }^fA_p(a_3,b_3)}\int_{A_p(a_3,b_3)}|{\rm hess }~\theta|^2 d{\rm v}^f\notag\\
&<\Psi(\epsilon,\omega, \delta;A,a_1,b_1,a_2,b_2,a_3,b_3,a,b)\label{hessian-est},
\end{align}
where $A_p(a_3,b_3)$ is an even smaller annulus in $A_p(a_2,b_2)$.
\end{lem}

\begin{proof}
Let   $h$ be the $f$-harmonic function  constructed  in (\ref{f-harmonic-radial}) in Section 2 and $\theta_1=\langle X,\nabla h\rangle$.  Then
\begin{align}
\nonumber \nabla \theta_1=\langle\nabla X, \nabla h\rangle+\langle X,{\rm hess }~h\rangle,
\end{align}
It follows
\begin{align}
\nonumber &\int_{A_p(a_2,b_2)}|\nabla\theta_1-X|^2d\text{v}^f\\
&\leq 2\int_{A_p(a_2,b_2)}(\langle\nabla X, \nabla h\rangle^2d{\rm v}^f+\langle X,{\rm hess }~h-g\rangle^2)d{\rm v}^f.\notag
\end{align}
Thus by  (\ref{conditions}) and Lemma \ref{harmonic-estimate-annual-2}, we  get
\begin{align}\label{gradient-est}
\frac{1}{{\rm vol }^f(A_p(a_2,b_2))}\int_{A_p(a_2,b_2)}|\nabla\theta_1-X|^2 d{\rm v}^f
< \Psi.\end{align}

Let $\theta$   be  a solution  of equation,
\begin{align}
\Delta^f\theta=0,~{ \rm in}~A_p(a_2,b_2), \end{align}
  with $\theta=\theta_1$   ~on  $\partial A_p(a_2,b_2)$.
Then
\begin{align}
&\int_{A_p(a_2,b_2)}(\langle \nabla\theta-\nabla\theta_1,X\rangle+(\theta-\theta_1){\rm div } X)d{\rm v}^f\notag\\
&=\int_{A_p(a_2,b_2)}{\rm div } ((\theta-\theta_1)X)d{\rm v}^f= \int_{A_p(a_2,b_2)}  (\theta-\theta_1)\langle \nabla f, X\rangle d{\rm v}^f.\notag
\end{align}
It follows
\begin{align}\label{divergence}
\int_{A_p(a_2,b_2)}\langle \nabla\theta-\nabla\theta_1,X\rangle d{\rm v}^f< \Psi.
\end{align}
On the other hand, since
\begin{align}
\nonumber \int_{A_p(a_2,b_2)}\langle \nabla\theta_1-\nabla \theta, \nabla\theta\rangle d{\rm v}^f=\int_{A_p(a_2,b_2)}(\theta-\theta_1)\Delta^f\theta d{\rm v}^f=0,
\end{align}
 we have
 \begin{align}
 \int_{A_p(a_2,b_2)} |\nabla \theta|^2 d{\rm v}^f=\int_{A_p(a_2,b_2)}  \langle \nabla\theta,\nabla\theta_1\rangle d\text{v}^f.\notag
\end{align}
By the H\"older inequality,   we get
$$\int_{A_p(a_2,b_2)} |\nabla \theta|^2 d{\rm v}^f\le \int_{A_p(a_2,b_2)} |\nabla \theta_1|^2 d{\rm v}^f<C.$$
Hence,
\begin{align}
\nonumber & \int_{A_P(a_2,b_2)}\langle \nabla \theta-X\rangle^2d{\rm v}^f\\
&=\int_{A_p(a_2,b_2)}(|\nabla \theta|^2+|X|^2-2\langle \nabla\theta ,X\rangle)d{\rm v}^f\notag\\
\nonumber &=\int_{A_p(a_2,b_2)}(\langle\nabla \theta, \nabla\theta_1\rangle+|X|^2-2\langle \nabla\theta ,X\rangle)d{\rm v}^f&\\
&=\int_{A_p(a_2,b_2)}(\langle \nabla\theta_1-X,\nabla\theta\rangle+\langle X,X-\nabla\theta_1\rangle+\langle X,\nabla\theta_1-\nabla \theta\rangle)d{\rm v}^f.&
\end{align}
 Therefore, combining (\ref{conditions}) and (\ref{divergence}),
 we derive (\ref{gradient-est-app}) immediately.

To get (\ref{hessian-est}), we choose a cut-off function  which is $\phi$ supported in $A_p(a_2,b_2)$ with bounded gradient and
$f$-Lapalace as in Lemma  \ref{cut-off} in Section 1.   Then  by
 the Bochner identity,  we have
\begin{align}
 \int_{A_p(a_2,b_2)}\frac{1}{2}\phi\Delta^f|\nabla \theta|^2d{\rm v}^f=\int_{A_p(a_2,b_2)}\phi(|{\rm hess }~\theta|^2+{\rm Ric }(\nabla \theta,\nabla \theta)) d{\rm v}^f.\notag
\end{align}
Since
\begin{align}
 \int_{A_p(a_2,b_2)}\frac{1}{2}\phi\Delta^f|X|^2d{\rm v}^f
 =-\int_{A_p(a_2,b_2)}\langle\nabla \phi,\langle X,\nabla X\rangle\rangle d{\rm v}^f,\notag
\end{align}
 we obtain
\begin{align}
\nonumber \int_{A_p(a_2,b_2)}\phi(|{\rm hess }\theta|^2  d{\rm v}^f& < \int_{A_p(a_2,b_2)}\frac{1}{2}\phi\Delta^f(|\nabla \theta|^2-|X|^2)d{\rm v}^f\\
&+\Psi(\epsilon,\omega, \delta;A,a_1,b_1,a_2,b_2,a_3,b_3,a,b).&
\end{align}
Therefore,  using  integration  by parts, we  derive  (\ref{hessian-est}) from  (\ref{gradient-est-app}).
\end{proof}

Next, we generalize Proposition \ref{proof-splitting} to  the case without the  assumption of  the existence of an almost line.

\begin{lem}\label{split-integral}
Let $(M,g)$ be a Riemannian manifold  which satisfies  (\ref{be-curvature-condition}).  Let  $h^+$ be
a  $f$-harmonic function  which satisfies
\begin{align}\label{gradient-C^0-app}|\nabla h^+|\leq c(n,\Lambda,A),
\end{align}
\begin{align}
\label{gradient-condition-app}
\frac{1}{\rm {vol}^f (B_p(1))}|\int_{B_p(1)}|\nabla h^+|^2-1|  d{\rm v}^f< \delta,
\end{align}
\begin{align}\label{hessian-condition-app}\frac{1}{{\rm vol}^f (B_p(1))}\int_{B_p(1)}|{\rm hess }~h^+|^2  d\text{v}^f < \delta.
\end{align}
 Then there  exists  a $\Psi(\delta; A,\Lambda,n)$ Gromov-Hausdorff approximation from $B_p(\frac{1}{8})$ to  $B_{(0\times x)}(\frac{1}{8})\subset\mathbb{R}\times X$.
\end{lem}

The proof of Lemma  \ref{split-integral} depends on the following fundamental lemma  which is in fact a consequence of Theorem 16.32 and Lemma 8.17 in [Ch1].

\begin{lem}\label{level-set-function}
Under the condition (\ref{be-curvature-condition}), for a $f$-harmonic function $h^+$   which satisfies  (\ref{gradient-C^0-app}), (\ref{gradient-condition-app}) and  (\ref{hessian-condition-app}) in  $B_p(1)$,
 there exists a Lipschitz  function $\rho$ in $B_p(\frac{1}{4})$ such that  $|h^+-\rho|< \Psi$ and
\begin{align}\label{app-dis}
 ||\rho(z)-t|-d(z,\rho^{-1}(t))|< \Psi.
\end{align}
\end{lem}

\begin{proof} First, we  notice  that the following Poincar\'{e} inequality holds for any $C^1$-function $h$,
\begin{align}\label{Poincare}
&\frac{1}{{\rm vol}^f(B_p(\frac{1}{2}))}\int_{B_p(\frac{1}{2})}|h-a|^2 d{\rm v}^f\notag\\
&\leq c(n,\Lambda, A)\frac{1}{{\rm vol}^f(B_p(1))}\int_{B_p(1)}|\nabla h|^2 d{\rm v}^f,
\end{align}
where $$a=\frac{1}{{\rm vol}^f(B_p(\frac{1}{2}))}\int_{B_p(\frac{1}{2})}h d{\rm v}^f.$$
This is in fact  a consequence of Lemma \ref{segment-inequ}  by applying  the function $e$ to $|\nabla h|^2$, because
\begin{align}
& \frac{1}{{\rm vol}^f(B_p(\frac{1}{2}))}  \int_{B_p(\frac{1}{2})}|h(x)-a|^2 d{\rm v}^f\notag\\
&=  \frac{1}{{\rm vol}^f(B_p(\frac{1}{2}))}  \int _{B_p(\frac{1}{2})} d{\rm v}_x^f[   \frac{1}{{\rm vol}^f(B_p(\frac{1}{2}))}  \int_{B_p(\frac{1}{2})}(h(x)-h(y))d {\rm v}_y^f]^2
\notag\\
& \leq    \frac{1}{{\rm vol}^f(B_p(\frac{1}{2}))}  \int_{B_p(\frac{1}{2})}    \frac{1}{{\rm vol}^f(B_p(\frac{1}{2}))}  \int_{B_p(\frac{1}{2})}(h(x)-h(y))^2 d {\rm v}_x^f d{\rm v}_y^f\notag\\
  &\leq    \frac{1}{{\rm vol}^f(B_p(\frac{1}{2}))}  \int_{B_p(\frac{1}{2})}   \frac{1}{{\rm vol}^f(B_p(\frac{1}{2}))}  \int_{B_p(\frac{1}{2})}\int_0^{d(x,y)}|\nabla h((\gamma(s))|^2 d{\rm v}_x^f d {\rm v}_y^f\notag\\
  &\leq c(n,\Lambda, A)   \frac{1}{{\rm vol}^f(B_p(1))}  \int_{B_p(1)}|\nabla h|^2 d{\rm v}^f.\notag
\end{align}
Thus by taking $h=|\nabla h^+|^2$,  we get from  (\ref{gradient-C^0-app})-(\ref{hessian-condition-app}),
 \begin{align}\label{gradient-condition-2-app}\frac{1}{{\rm vol}^f (B_p(\frac{1}{2}))}\int_{B_p(\frac{1}{2})}||\nabla h^+|^2-1| d{\rm v}^f< \Psi.
\end{align}

 Next we  apply Theorem 16.32 in [Ch1]  to $h^+$ with the condition  (\ref{gradient-C^0-app}), (\ref{gradient-condition-app}) and (\ref{gradient-condition-2-app}).  We suffice to check  a doubling condition for the measure $d{\rm v}^f$  and an $(\epsilon,\delta)$-inequality.   The  $(\epsilon,\delta)$-inequality says,  for any $\epsilon,\delta>0$  and  two points $x,y\in M$ with $d(x,y)=r$, there exist  $C_{\epsilon,\delta}$  and another two points  $x', y'$  with  ${\rm d}(x', x)\leq \delta r $ and $d(y', y)\leq \delta r $,  respectively such that
\begin{align}
&F_{\phi,\epsilon}(z_1',z_2')\leq \frac{C_{\epsilon,\delta}r}{{\rm vol}^f(B_{z_1}((1+\delta)(1+2\epsilon)r))}\int_{B_{z_1}((1+\delta)(1+2\epsilon)r)}\phi d{\rm v}^f,
\end{align}
    where
 $$F_{\phi,\epsilon}(x,y)=\inf\int_0^l \phi(c(s))ds, ~\forall ~\phi(\ge 0)\in ~C^0(M),$$
  and the  infinimum   takes  among all curves
 from $x$ to $y$ with length   $l\leq (1+\epsilon){\rm d}(x,y)$.
The doubling condition follows from  Volume Comparison Theorem \ref{volume-comparison},  and   $(\epsilon,\delta)$-inequality follows from Volume Comparison Theorem \ref{volume-comparison}
and the segment inequality in Lemma \ref{equ-seg}. Thus
we can construct a Lipschitz function $\rho$  from $h^+$ such that   $|h^+-\rho|\leq \Psi$. Moreover, by  Lemma 8.17 in  [Ch1],
we get (\ref{app-dis}).
\end{proof}

\begin{proof}[Proof of  Lemma \ref{split-integral}]
As  in the proof  of Proposition \ref{proof-splitting},  we define $X=(h^+)^{-1}(0)$  and the map $u$  by
\begin{align}
\nonumber u(q)=(h^+(q),x_q),
\end{align}
where $x_q$ is the nearest point in $X$ to $q$.  To show that $u$ is a Gromov-Hausdorff approximation,  we  shall use  Lemma \ref{cheeger-lemma}.  In fact,  by  (\ref{app-dis}) in Lemma \ref{level-set-function},
we see
\begin{align}\label{gradient-instead}
||h^+(z)-t|-{\rm d}(z,(h^+)^{-1}(t))|< \Psi.
\end{align}
 Then  instead  of (\ref{triangular-equ})  by (\ref{gradient-instead}),   Lemma \ref{cheeger-lemma} is still  true since  (\ref{hessian-condition-app})  holds [C2].
 Hence  the proof  in  Proposition \ref{proof-splitting} works for  Lemma \ref{split-integral}.
\end{proof}

Now we begin to prove (\ref{J-invariant}) in Section 6.  Let  $(M,g)$ be  a K\"{a}hler manifold which satisfies  (\ref{condition-1-regularity}).   Let $B_p(l)\subset M$  and $B_{(0\times x)}(l)\subset \mathbb{R}^{2n-2}\times X$ be   two $l$-radius  distance  balls as in Section 6.   Then

\begin{prop}\label{J-invariant-property}
 Suppose that
\begin{align}\label{cone-condition-app}
{\rm d}_{GH}(B_p(l),B_{(0\times x)}(l)) <\eta.\end{align}
  Then either $B_p(\frac{1}{8})$ is close to an Euclidean ball in  the Gromov-Hausdorff topology or for a suitable choice of the orthogonal coordinates  in  $\mathbb{R}^{2n-2}$, the map $\Phi=(h_1,...,h_{2n-1})$ constructed in Section 5 satisfies
\begin{align}\label{almost-complex}
\frac{1}{{\rm vol}^f B_p(1)}\int_{B_p(1)}|\nabla h_{n-1+i}-\mathbf{J}\nabla h_i|^2 d{\rm v}^f <\Psi(\tau,\eta,\frac{1}{l};v).
\end{align}
\end{prop}

\begin{proof}
Roughly speaking, if the space spanned by $\nabla h_i$ is not almost $\mathbf{J}$ invariant, we can find a vector field nearly perpendicular to these $\nabla h_i$,  and  it  satisfies  the condition (\ref{conditions})  in Lemma \ref{almost-gradient-vector}.
Then  by Lemma \ref{split-integral},  $B_p(1)$ will be almost split off along a new line.  This  implies that $B_p(\frac{1}{8})$ is  close to an  Euclidean ball.

 Let  $V$  be a $(4n-4)$-dimensional line space  spanned by $\nabla h_i,\mathbf{J}\nabla h_i$  with the $L^2$-inner product,
 $$(b_i,b_j)_{L^2}=\int_{B_p(1)}  \langle b_i, b_j\rangle d\text{v}.$$
 Then $\mathbf{J}$ induces an complex structure on $V$ such that the inner product is $\mathbf{J}$-invariant. We introduce a distance  in   Grassmanian $G(2n,k)$ as follows,
\begin{align}
{\rm d}(\Lambda_1,\Lambda_2)^2=\mathop{\sum}_j \|{\rm pr}_{\Lambda_2}^{\perp}(e_j)\|_{L^2}^2
\end{align}
for any  two $k$-dimensional subspaces  $\Lambda_1,\Lambda_2$ in $\mathbb{R}^{2n}$,
 where $e_i$ is an  unit orthogonal basis of $\Lambda_1$ and $pr_{\Lambda_2}^{\perp}$ is the compliment of orthogonal  projection to $\Lambda_2$.
 First  we suppose that
$${\rm d}(W,\mathbf{J} W)^2 < \Psi,$$
where $W={\rm span}\{\nabla h_i|i=1,2,...,2n-2\}$.   Then by  the Gram-Schmidt process,  one can find a unit orthogonal basis $w_i$ of $W$ such that
 $$\|\mathbf{J}w_i-w_{n-1+i}\|_{L^2} < \Psi.$$
 It is equivalent to that there exists a matrix $a_{ij}\in GL(2n-2,\mathbb{R})$ which is nearly orthogonal such that
$$w_i=\Sigma_j a_{ij}\nabla h_j. $$
Thus  by  changing  an orthogonal basis in $\mathbb{R}^{2n-2}$,  (\ref{almost-complex})  will be  true.

Secondly,   we suppose that
$${\rm d}(W,\mathbf{J}W) >\delta_0.$$
This implies that   there exists some $j$ such that
  $$ \|{\rm pr}_W^\perp(\mathbf{J}\nabla h_i)\|_{L^2}=\|\mathbf{J}\nabla h_i-{\rm pr}_W(\mathbf{J}\nabla h_i) \|_{L^2}>  \frac{\delta_0}{2n}.$$
Let
\begin{align}\label{vector}
X=\frac{{\rm pr}_W^\perp(\mathbf{J}\nabla h_i)}{\|{\rm pr}_W^{\perp}(\mathbf{J}\nabla h_i)\|_{L^2}}
\end{align}
Then ${\rm pr}_W^\perp(\mathbf{J}\nabla h_i)$ is perpendicular to  $W$ with $\|{\rm pr}_W^\perp(\mathbf{J}\nabla h_i) \|_{L^2}=1$ and it  satisfies  the condition (\ref{conditions})  in Lemma \ref{almost-gradient-vector}.
  Thus  we see that there
exists a $f$-harmonic function $\theta$  which satisfies the conditions (\ref{gradient-C^0-app}), (\ref{gradient-condition-app}) and  (\ref{hessian-condition-app}) in Lemma \ref{split-integral}.   As a consequence,   $B_p(\frac{1}{8})$  will  almost  spilt off along a new  line associated to the coordinate function $\theta$.   Since $X\in W^{\perp}$,
$B_p(\frac{1}{8})$  in fact  split off $\mathbb{R}^{2n-1}$ almost.
But  the late implies that $B_p(\frac{1}{8})$ is close to an Euclidean ball in  the Gromov-Hausdorff
topology  by using a topological argument as in   Theorem 6.2 in  [CC2] or by   the following Proposition \ref{prop-even-dim}
for K\"ahler manifolds.

\end{proof}

\begin{prop}\label{prop-even-dim}
Let $Y$ be  a  limit space  of a sequence of K\"ahler manifolds in Theorem  \ref{dimension-k}.
 Then
 $$\mathcal{S}(Y)=\mathcal{S}_{2k+1}=\mathcal{S}_{2k}.$$
\end{prop}

\begin{proof}

We  suffice  to show that if a tangent cone $T_yY$ at a point $y\in Y$ can split off $\mathbb{R}^{2k+1}$,  $T_yY$ can split off $\mathbb{R}^{2k+2}$.   Let  $h_i$  be   $2k+1$  $f$-harmonic functions  which  approximate $2k+1$  distance functions with
different directions as constructed in Section 2 and Section 3.   Then as in the proof of Proposition \ref{J-invariant-property},   we consider a  linear space $V=\text{span}\{\nabla h_i,\mathbf{J}\nabla h_i\}$ with $L^2$-inner product.  Since the dimension of $W=\text{span}\{\nabla h_i\}$ is odd, we have
$${\rm d}(W,\mathbf{J}W)\geq 1.$$
Thus  $T_yY$ will  split off a new  line.  The  proposition  is proved.

\end{proof}

\vskip3mm

%\newpage

\end{document}